\documentclass[a4paper]{siamonline250211} 
\usepackage{damacros}
\usepackage{mwe}
\usepackage{pict2e}
\makeatletter
\DeclareRobustCommand{\bigtimes}{%
  \mathop{\vphantom{\sum}\mathpalette\@bigtimes\relax}\slimits@
}
\newcommand{\@bigtimes}[2]{\vcenter{\hbox{\make@bigtimes{#1}}}}
\newcommand{\make@bigtimes}[1]{%
  \sbox\z@{$\m@th#1\sum$}%
  \setlength{\unitlength}{\wd\z@}%
  \begin{picture}(1,1)
  \roundcap
  \linethickness{.17ex}
  \Line(.1,.1)(.9,.9)
  \Line(.1,.9)(.9,.1)
  \end{picture}%
}
\makeatother

\newsiamthm{problem}{Problem}
\crefname{problem}{Problem}{Problems}
\Crefname{problem}{Problem}{Problems}
\renewcommand{\phi}{\varphi}

\newcommand{\epsi}{\varepsilon}

\newcommand{\TheTitle}{Neural Field Equations with random data}
\newcommand{\TheShortTitle}{Neural field equations with random data}
\newcommand{\TheAuthors}{D. Avitabile, F. Cavallini, S. Dubinkina, G.~J. Lord}
\headers{\TheShortTitle}{\TheAuthors}

\title{{\TheTitle}}
 
\author{%
    Daniele Avitabile\thanks{%
    Amsterdam Center for Dynamics and Computation,
    Department of Mathematics,
    Vrije Universiteit Amsterdam, 
    De Boelelaan 1111,
    1081 HV Amsterdam, The Netherlands.
    MathNeuro Team,
    Inria branch of the University of Montpellier,
    860 rue Saint-Priest
    34095 Montpellier Cedex 5
    France.
    (\email{d.avitabile@vu.nl}).
  }
  \and Francesca Cavallini\thanks{%
    Amsterdam Center for Dynamics and Computation,
    Department of Mathematics,
    Vrije Universiteit Amsterdam, 
    De Boelelaan 1111,
    1081 HV Amsterdam, The Netherlands.
  }
  \and Svetlana Dubinkina\footnotemark[2]
  \and Gabriel~J. Lord\thanks{%
    Department of Mathematics,
    Radboud University, 
    Postbus 9010,
    6500 GL Nijmegen,
    The Netherlands.
  }
}

\graphicspath{{Figures/}}

\renewtheorem{proposition}{Proposition}[section]

\renewtheorem{corollary}{Corollary}[section]

\renewtheorem{lemma}{Lemma}[section]

\newcommand{\revision}[1]{\textcolor{black}{#1}}

\begin{document} 
\maketitle

\begin{abstract} 
  We study neural field equations, which are prototypical models of large-scale
  cortical activity, subject to random data. We view this spatially-extended,
  nonlocal evolution equation as a Cauchy problem on abstract Banach spaces, with
  randomness in the synaptic kernel, firing rate function, external stimuli, and
  initial conditions. We determine conditions on the random data that guarantee
  existence, uniqueness, and measurability of the solution for uncertainty
  quantification (UQ), and examine the regularity of the solution in relation to the
  regularity of the inputs. We present results for linear and nonlinear neural
  fields, and for the two most common functional setups in the numerical analysis of
  this problem. In addition to the continuous problem, we analyse in abstract form
  neural fields that have been spatially discretised, setting the foundations for
  analysing UQ schemes.
 \end{abstract}

% \da[inline]{Get urls for Arxiv submissions of both papers, and make sure they are
% cited in the cross-references of each bibliography}
% \da[inline]{Prepare GitHub repository, add code repos in both papers.}
% \da[inline]{Understand where we use piecewise smooth $D$ \cref{hyp:domain}, in both papers.}
% \da[inline]{Acronym: check for consistency and first-place definition}

 \section{Introduction}\label{sec:intro} Modelling and forecasting
 brain dynamics is a fundamental challenge in biology.
 %Driven partly by experimental advances, and partly by
 %the development of dedicated mathematical techniques, the mathematical and
 %computational neuroscience communities continue to develop methods and models to
 %elucidate the brain's intricate network structure and the rich
 %multi-scale dynamics taking place in the cortex.
 Although voltage dynamics of single cells are well described by models of Hodgkin--Huxley
 type,
 %This formalism generates a versatile hierarchy of biophysical models,
 %providing descriptions for a great variety of ion channels, currents, and neuronal
 %rhythms 
 \cite{izhikevichDynamicalSystemsNeuroscience2006,ermentroutMathematicalFoundationsNeuroscience2010,
 bressloffWavesNeuralMedia2014, gerstnerNeuronalDynamicsSingle2014,
coombesNeurodynamicsAppliedMathematics2023}, 
 the picture complicates considerably for neuronal ensembles. Research efforts have
 been made to couple and simulate massive numbers of Hodgkin--Huxley or spiking
 neurons with anatomical realism
 \cite{markramReconstructionSimulationNeocortical2015,schmidtMultiscaleLayerresolvedSpiking2018},
 but analysing and simulating wide cortical sheets with microscopic detail continues
 to be a challenge. An alternative strategy is to trade biological realism at the
 microscale for mathematical tractability at the macroscale; this gives rise to
 models representing the cortex as a continuum, and cortical activity as a scalar
 field. These \textit{neural field equations} (see \cref{eq:NFDeterministic}),
 albeit phenomenological in nature, support
 cortical patterns observed in experiments (see~\cite{coombes2014neural} for a
 monograph); in addition, neural fields expose inputs and outputs (in both functional
 and parametric form) that can be fit to data 
 \cite{huangSpiralWavesDisinhibited2004a, schiffKalmanFilterControl2008,
   pinotsisNeuralFieldsSpectral2011,gonzalez-ramirezl2015}.
 %It is therefore relevant and useful to explore how such model propagates uncertainties in the inputs.
 
 This paper sets the theoretical foundations to quantify uncertainty in
 neurobiological cortical models at the macroscale, by studying \textit{neural field
 equations} such as  \eqref{eq:NFDeterministic} subject to random data. The neural
 field equation does not fit in the traditional ODE/PDE framework of e.g.
 \cite{smithUncertaintyQuantificationTheory2014,
 sullivanIntroductionUncertaintyQuantification2015,
 ghanemHandbookUncertaintyQuantification2017,calvettiBayesianScientificComputing2023},
 as they are integro-differential equations. Instead they require a dedicated treatment,
 guided by the existing literature on PDEs with random data. Furthermore,  numerical
 schemes for forward and inverse UQ in PDEs rely on well-posedness and regularity
 results for continuous and semidiscrete PDEs with random data and such
 characterisation is absent for spatially-extended cortical models and we fill this
 gap.

 We defer to later a discussion on the applicability of our results to more (and less)
 realistic models, and we now introduce neural fields, discuss their input data, and
 present motivating numerical simulations.
 
 \subsection*{Deterministic model, and sources of randomness} The simplest and most
 popular neural field is the following integro-differential equation: 
   \begin{equation}
   \label{eq:NFDeterministic} 
   \begin{aligned} 
   & \partial_t u(x,t) = -u(x,t) + \int_D
     w(x,x') f(u(x',t))\, dx' + g(x,t), && (x,t) \in D \times [0,T], \\ 
  & u(x,0) = v(x),                      && r \in D. 
  \end{aligned} 
  \end{equation} 
 Proposed
 independently by Wilson and Cowan~\cite{wilson1973mathematical}, and by
 Amari~\cite{amari1977dynamics}, the neural field presented above  is a
 spatially continuous, coarse-grained model of cortical activity. The cortex $D$ is
 a compact in $\RSet^d$, and the state variable $u(x,t)$ models the voltage of a
 neural patch at time $t$ and position $x \in D$.
 The function $w(x,x')$ is the \textit{synaptic kernel}, modelling the strength of
 synaptic connections from point $x'$ to point $x$ in the tissue. Some connections
 run within the cortex (through the grey matter), while others are bundled in fibres
 that leave and re-enter the cortex over long-range distances (through the white
 matter). Depending on the scale at which the model is posed, the function $w(x,x')$
 encodes either type of connection, or both.
 Nonlocal contributions are weighed by the synaptic kernel, and regulated by the
 nonlinear function $f$, which models the neuronal population's \textit{firing
 rate}, and it is typically a sigmoid with variable steepness. The functions $g$ and
 $v$ represent the external inputs and the initial voltage, respectively.

% Since their introduction, neural fields proved to be a versatile tool to investigate
% the mechanisms of pattern formation in brain dynamics.
% System
% \cref{eq:NFDeterministic} is amenable to nonlinear analysis
% \cite{Ermentrout.1998qno, bressloffSpatiotemporalDynamicsContinuum2012,
% coombes2014neural}, functional analysis
% \cite{Potthast:2010kb,faugeras2008absolute}, numerical analysis
% \cite{laingSpiralWavesNonlocal2005, rankinContinuationLocalisedCoherent2014,
% laing2014numerical, limaNumerical2022, avitabileProjectionMethodsNeural2023}, 
% which predict the occurrence of patterns and link them
% to macroscopic observables, or population properties
% \cite{laingSpiralWavesNonlocal2005,kilpatrick2013wandering,bressloffWavesNeuralMedia2014} and has been linked directly to data
% \cite{Potthast:2009bd,NakamuraInverseMod}.

% The mathematical neuroscience community has extended neural fields in various ways,
% for instance by introducing multiple neuronal populations
% \cite{pintoSpatiallyStructuredActivity2001} or synaptic delays
% \cite{Faye.2010,huttDistributedNonlocalFeedback2013,Gils.2013,meijerTravellingWavesNeural2014a}.
% Meanwhile, spatially-discrete versions of these models have been studied prominently
% in the area of computational neuroscience, where they are known as \textit{neural
% mass} or \textit{connectomic} models).

 The well-posedness of deterministic neural fields has been proved with functional
 analytical methods, viewing the problem as an ODE on a Banach space $\XSet$. In this
 area, two groups worked independently and simultaneously on the cases $\XSet =
 C(D)$~\cite{Potthast:2010kb}, the space of continuous functions on $D$, and
 $\XSet = L^2(D)$~\cite{Faugeras:2009gn}, the space of square-integrable functions on $D$. In
 a similar spirit, the problem with delays has been studied with fixed-point arguments
 \cite{Faye.2010}, and with a bespoke approach based on sun-star calculus
 \cite{Gils.2013}.  

 To date, noise in neural fields has been introduced solely in the form of a
 stochastic forcing $g$, thus turning the problem into a stochastic
 integro-differential equation
 \cite{kilpatrick2013wandering,Kuehn:2014im,faugerasStochastic2015,
 maclaurin2020wandering,limaNumerical2022,krugerWellposednessStochasticNeural2018,NakamuraInverseMod,
 otsetovaErgodicityStochasticNeural2025}.
 This line of work differs from the one
 taken here, where we deal with equations with random data, multiple sources of noise,
 and numerical analysis.
 %Since we deal with deterministic evolution equations subject
 %to random data, the technical apparatus developed for stochastic differential
 %equations is not applicable. 

% Further, inverse problems to estimate the synaptic kernel and voltage have been
% studied by Potthast and co-workers \cite{Potthast:2009bd,NakamuraInverseMod}. These
% papers overcome ill-posedness of the inverse problem with classical regularisation
 % methods. With the present paper,
 We aim to open up the possibility of treating
 forward and inverse problems within a Bayesian framework. In particular, we formalise
 neural fields as Cauchy problems in which the functions $w$, $f$, $g$, and $v$ are
 independent but concurrent random fields in a suitable Bochner space, we define
 solutions to the problem, and study their properties.
 \begin{figure}
   \centering
   \includegraphics{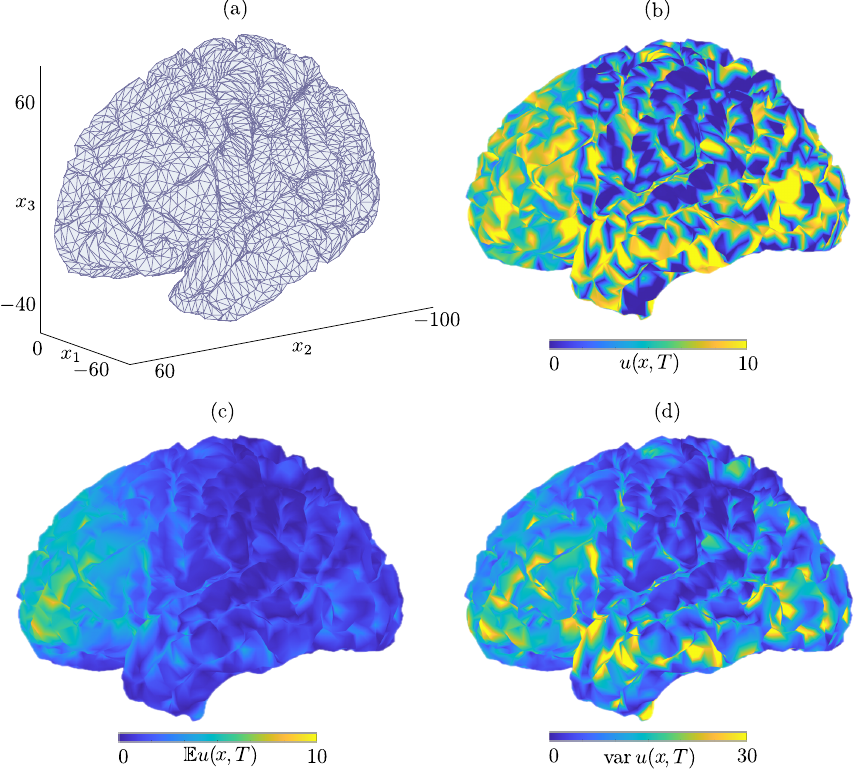}
   \caption{Example of a forward UQ problem for the neural field equation
      \cref{eq:NFDeterministic}, with random data
      \crefrange{eq:sigmoidalExample}{eq:kernelExample}. (a) The problem is discretised in space
      using a triangulated mesh with $n = 10242$ nodes, taken from the Human Connectome
      Project~\cite{marcusInformaticsDataMining2011}, and features 39459 random parameters
      in the definitions of $f$, $g$, and $w$ (see main text). 
      A Finite-Element
      scheme~\cite{avitabile2024} is combined with standard adaptive Runge-Kutta $4$th
      order scheme to timestep the problem up to time $T$. (b) Sample of the final
      solution $u(x,T)$
      \href{https://figshare.com/articles/media/Animations_for_the_paper_Neural_fields_with_random_data_/27559950?file=50168907}{(link to an animation of $u(x,t)$ for $t \in [0,T]$)}. 
      (c) Estimate of the expectation of $u(x,T)$, computed using
      $100$ Monte Carlo samples
      \href{https://figshare.com/articles/media/Animations_for_the_paper_Neural_fields_with_random_data_/27559950?file=50168910}{(link
      to animation)}. 
      (d) Estimated variance of $u(x,T)$
      \href{https://figshare.com/articles/media/Animations_for_the_paper_Neural_fields_with_random_data_/27559950?file=50168913}{(link
      to animation)}. The random
      parameters are uniformly and independently distributed (see main text) on the
      intervals $[\alpha_f,\beta_f]=[0,3]$, $[\alpha_\mu,\beta_\mu]=[10,15]$
      $[\alpha_c,\beta_c]=[1,10]$, and $[\alpha_w,\beta_w]=[0,3]$, respectively. Other parameters: $h= 0.5$, $A= 10$,
      $(y_{1},y_{2},y_{3})=(-27,70,43)$, $\sigma =(30,1,30)$, $\sigma_w=10/3$, $\rho=\sqrt{\sigma_w \ln 10}$, and $T = 10$.
   }
   \label{fig:monteCarloSimulation}
 \end{figure}

 \subsection*{A motivating example}
 To give an example of the computations targeted by our work, we describe the numerical
 experiment in \cref{fig:monteCarloSimulation}, showing computations of a forward UQ
 calculation for model \cref{eq:NFDeterministic}. The example studies cortical responses to a
 sharp, localised stimulus imparted on the prefrontal cortex as may be done, for
 instance, in Transcranial Magnetic Stimulation
 ~\cite{gomez-tamesReviewBiophysicalModelling2020}. We examine how the neural field
 model \cref{eq:NFDeterministic} propagates uncertainties in firing rate, synaptic
 kernel, and external stimulus.

 The domain $D$ is a triangulated surface in
 $\RSet^3$ representing the left hemisphere of human cortex (mesh downloaded from the
 Human Connectome Project~\cite{marcusInformaticsDataMining2011}) comprising $n$
 vertices $\{ \xi_1,\ldots,\xi_n\} =: \Xi$. The initial condition of the problem is
 deterministic, and set to $v(x) \equiv 0$. Random data in the neural field arises
 via parametrised random fields, which we introduce informally for the time being.
 The firing rate is a sigmoidal
\begin{equation}\label{eq:sigmoidalExample}
  f(u)=\frac{f^*}{1+e^{-\mu^*(u-h)}}, 
  \qquad  f^* \sim \mathcal{U}[\alpha_f,\beta_f],
  \qquad  \mu^* \sim \mathcal{U}[\alpha_\mu,\beta_\mu],
\end{equation}
with random maximum value $f^*$ and steepness parameter $\mu^*$, and
deterministic activation threshold $h$. The external stimulus is a localised pulse
centred initially at $(y_{1}, y_{2}, y_{3})$, and travelling along the horizontal axis with random
speed $c^* \sim \mathcal{U}[\alpha_c,\beta_c]$, 
\begin{equation}\label{eq:forcingExample}
    g\bigl((x_1,x_2,x_3),t\bigr)=
    \frac{A}{\cosh((x_2-y_{2}+t c^*)/\sigma_{2})^2}
  \exp \Biggl[-\frac{(x_1-y_{1})^2}{2\sigma^2_{1}} -
  \frac{(x_3-y_{3})^2}{2\sigma^2_{3}}\Biggr],
\end{equation}
in which $A$ is the forcing amplitude, and $\sigma =
(\sigma_{1},\sigma_{2},\sigma_{3})$ its
characteristic length scales.

The synaptic kernel $w(x,x')$ is a random perturbation of a deterministic kernel
$k(x,x')$. More precisely, the deterministic kernel $k(x,x')$ depends on the
Euclidean distance between points, and is set to zero for long-range interactions, as
follows
\begin{equation}\label{eq:kernelExample}
  k(x,x') = K(\|  x-x' \|_2), \qquad 
  K(x) = e^{-x^2/\sigma_w} 1_{[-\rho,\rho]}(x)
\end{equation}
The random perturbation to $k$ is obtained via piecewise-constant functions with random values, and
supported in a small neighbourhood of each vertex of the triangulation at which $k$
is nonzero,

\begin{equation}\label{eq:kernelPertExample}
  w(x,x') = k(x,x') + \sum_{(\xi_i,\xi_j) \in \Lambda} W^*_{ij} \,
  % 1_{B^\epsi_{\xi_i}}(x)
  % 1_{B^\epsi_{\xi_j}}(x'),
  1_{B(\xi_i,\epsi)}(x)
  1_{B(\xi_j,\epsi)}\revision{(x')}
  \qquad  
  W^*_{ij} \stackrel{\text{i.i.d.}}{\sim} \mathcal{U}[\alpha_w,\beta_w],
\end{equation}
with $\Lambda =  (\Xi \times \Xi) \cap \supp(k)$ 
and $B(\xi,\epsi) = \{ x \in \RSet^3 \colon \| x -\xi \|_2<\epsi\}$. The parameter
$\epsi$ is chosen small enough to guarantee that, for any index $i$, each ball
$B(\xi_i,\epsi)$ does not
contain any vertex in the triangulation other than $\xi_i$. An alternative way to
make sense of the random perturbation to $k$ is at discrete level: the problem is discretised
using finite elements, and the kernel $k$ gives rise to a sparse finite element
matrix~\cite{avitabileProjectionMethodsNeural2023,avitabile2024}, which we perturb by
adding uniform, independently distributed random variables to each one of its nonzero
entries.

We time step the resulting ODE using a Runge-Kutta 4th order scheme with adaptive
stepsize. The forward UQ problem has $n= 10242$ unknowns, and 39459
random parameters, 3 of which come from $f$ and $g$, and the rest from the kernel.
In \cref{fig:monteCarloSimulation}(b) we show one realisation of
solution at final time $T=10$. We observe local regions of activity due to the wave
forcing $g$ as well as nonlocal regions given by random long-range synaptic
connections intensified by the random data in the kernel. The quantities of interest
of the problem are the mean and variance of the solution at final time. We
approximate them here via a Monte Carlo method with $100$ samples and we plot them in
\cref{fig:monteCarloSimulation}(c,d), respectively. The
paper~\cite{avitabile2024Stochastic} is a companion to the present one, and studies
stochastic collocation for a similar task. Examples such as this arise naturally in
applications, and we consider here the rigorous foundations
for their analysis by addressing the following questions: How should we define random
data for a neural field? Is problem \cref{eq:NFDeterministic} subject to random data
well-posed? In which function spaces do the solution and the random data live, and
what is the regularity of the solutions?  How do finite element or other spatial discretisations
affect this analysis (see \cref{thm:existenceNF_Pn})? 
%but there is at
%present no convergence theory for UQ schemes for neural field problems.

\subsection*{Summary of main results} With the view of addressing UQ problems in
 general form, we cast neural fields with random data as abstract nonlinear Cauchy
 problems, with random vectorfields, posed on an
 infinite-dimensional Banach space $\XSet \in \{ C(D), L^2(D)\}$. Our aim is to cover at once both
 functional settings available in the literature on deterministic neural fields, and
 on their numerical simulation. To this end, we build a theory which does not rely
 explicitly on the choice of the phase space: users are only required to check
 hypotheses that depend on $\XSet$ at the outset, and they can use the provided
 estimates in the corresponding natural norm thereafter.

 We envisage two use-cases for this work: one in which $f$ is nonlinear and bounded,
 and  one in which $f$ is linear. The latter is a good testing ground to develop
 UQ algorithms, and it is also seen in applications for which the dynamics of
 interest is a small perturbation around a rest state
 ~\cite{polyakovNeuralFieldTheory2024}. Even though the two cases require different
 hypotheses and technical treatment, we present results and proofs in parallel,
 whenever possible. In particular:
 \begin{enumerate}
   \item We outline hypotheses on the synaptic kernel, firing rate, external
     forcing, and initial conditions that guarantee the well-posedness of neural field
     problems with random data. Deterministic neural fields enjoy classical solutions,
     hence we seek for bounds in the strong norms on $C^r([0,T], \XSet)$ with $r = 0,1$.
     In the case of noisy data, well-posedness entails the existence and
     uniqueness of a strongly-measurable random variable $\omega \mapsto u(x,t,\omega)$
     taking values in $C^r([0,T],\XSet)$, and satisfying almost surely a random
     version of the neural field equation (see \cref{thm:existenceRealisationLNF}).

   \item We then look into regularity of the solutions. We initially consider
     non-parametric input data, in a suitable Bochner space, and prove that
     $L^p$-regular input data results in $L^p$-regular random solutions, that is, we
     determine conditions on the random data that guarantee $u \in
     L^p(\Omega,C^1([0,T],\XSet))$ (see \cref{thm:LpRegularity}).
     We then derive analogous results for parametric input data, in the form of
     finite-dimensional noise of arbitrary size
     (see \cref{lem:finDimNoise,cor:uLRho}). This is a useful characterisation when
     the input data comes, for instance, from truncated Karhunen-Loève expansions.
   \item
     Existing numerical schemes for deterministic neural fields 
     discretise space using a projector on a finite-dimensional subspace of
     $\XSet$~\cite{atkinson1997,avitabileProjectionMethodsNeural2023}. A-priori
     estimates for UQ schemes in PDEs require well-posedness and regularity results
     for the semi-discrete problem, but they are not available for neural fields.
     With the view of facilitating the numerical analysis of UQ schemes, we provide
     the well-posedness and regularity results mentioned above also for semi-discrete
     versions of neural fields, in abstract form, and for generic projectors
     (\cref{thm:LpRegularity_Pn,cor:uLRho_Pn}).
     These estimates are thus immediately applicable in neural field models
     discretised with collocation or Galerkin schemes, Finite-Elements or Spectral
     methods, as is done for instance in ~\cite{avitabile2024Stochastic}. With this
     approach, we treat at once schemes in both strong and weak form.
 \end{enumerate}
 
\subsection*{Overview} 
The rest of the manuscript is organised as follows. In \cref{sec:notation} we set the
notation and give preliminary definitions. \Cref{sec:problemRandomData} describe the
hypotheses and functional-analytic setup for individual realisations of the input
data. The main theorems on the well-posedness and $L^p$-regularity of solution
are given in \cref{sec:uniqueness}, together with a discussion on random Volterra
integral operators. Implications and results of the standard Finite Dimensional Noise
assumptions on our problem are discussed in \cref{sec:FinDimNoise}. \Cref{sec:projProblem} is devoted to proving the statements of the main theorems for
the projected version of problem. In \cref{sec:conclusions} we comment on the
applicability of our results to connectomic ODEs and other neurobiological networks,
and state future research directions.

%\section{Preliminary notions and notation}\label{sec:notation}
\section{Mathematical setting and notation}\label{sec:notation} 
%Henceforth we use $x$ and $x'$, as opposed to $r$ and $r'$, to denote points in $\RSet^d$.
%We use $\NSet_m$, where $m \in \NSet$ to denote the integers from $1$ to $m$.
We let $\NSet_m:=\{1,2,\ldots,m\}$. %, where $m \in \NSet$ to denote the integers from $1$ to $m$.
For a compact $J \subset \RSet$ and Banach space $\YSet$, we denote by $C^k(J,\YSet)$
the space of $k$-times continuously differentiable functions with %the norm 
\[
  \| f \|_{C^k(J,\YSet)} = \sum_{i \in \NSet_k} \| D^i f \|_\infty, \qquad \| f
  \|_{\infty} := \sup_{t \in J} \| f(t) \|_{\YSet}.
\]

Further, we denote by $BC^0(\RSet)$ or $BC(\RSet)$, the space of bounded and
continuous real-valued functions defined on $\RSet$, with norm $\| \blank
\|_\infty$. For a $k \geq 1$ we define
\[
  BC^k(\RSet) = \{ f \in BC^{k-1}(\RSet) \cap C^k(\RSet) \colon 
  f^{(k)} \in BC(\RSet) \},
\]
with norm $\| f \|_{BC^k(\RSet)} = \sum_{i \in \NSet_k}\| f^{(i)} \|_{\infty}$. 
% We will also extend this definition to $\YSet$-valued functions defined on $\Gamma
% \subset \RSet^m$, and use the symbol $BC^k(\Gamma,\YSet)$.

For fixed Banach spaces
$\XSet,\YSet$, we use $BL(\XSet,\YSet)$ to indicate the Banach space of bounded
linear operator on $\XSet$ to $\YSet$, with the standard operator norm
\[
  \| A \|_{BL(\XSet,\YSet)} = \sup_{x \in \XSet \setminus \{0\}} \frac{\| A x
  \|_\YSet}{\|  x \|_\XSet},
\]
and set $BL(\XSet) := BL(\XSet,\XSet)$. We also work with $K(\XSet) \subset
BL(\XSet)$, the subspace of compact linear operators on $\XSet$ to itself, with norm
$\| \blank \|_{BL(\XSet)}$.

We recall the two notions of \textit{random} and \textit{strong (or Bochner) random}
variable with values in a Banach space, which are relevant for this work. Consider a
probability space $(\Omega,\calF,\prob)$ and the measurable space
$(\YSet,\calB(\YSet))$, where $\YSet$ is a Banach space, and $\calB(\YSet)$ its Borel
$\sigma$-field. 

\begin{definition}[Random variable]\label{def:randomVar}
A mapping $u \colon \Omega \to \YSet$ is a \textit{$\YSet$-valued random variable} if
it is measurable from $(\Omega,\calF)$ to $(\YSet,\calB(\YSet))$, that is, the set
$\{ \omega \in \Omega \colon u(\omega) \in B \}$ belongs to $\calF$
for any $B \in \calB(\YSet)$.
\end{definition}

\begin{definition}[Strong random variable]\label{def:strongRandomVar}
A mapping $u \colon \Omega \to \YSet$ is a \textit{strongly $\YSet$-valued
  $\prob$-measurable random
variable} if it is the pointwise limit of $\prob$-simple functions, that is, there
exists a sequence $\{ u_n \}_{n \in \NSet}$ of functions
\[
  u_n(\omega) = \sum_{i=1}^{I_n} 1_{\Omega_i} (\omega) y_i, \qquad \Omega_i \in
  \calF, \quad y_i \in \YSet, \qquad  I_n \colon \NSet \to \NSet,  
\]
such that $\lim_{n \to \infty} \|  u(\omega) - u_n(\omega) \|_\YSet \to 0$ as
as $n \to \infty$ for $\prob$-almost all $\omega \in \Omega$. 
\end{definition}

When the probability measure $\prob$ is clear from the context, we write that
$u$ is a \textit{strongly $\YSet$-valued random variable}. When $\YSet=\RSet$, we
write that $u$ is a \textit{strongly $\prob$-measurable random variable}, or a
\textit{strong random variable}.

\begin{remark}\label{rem:strongVsWeak}
If $\YSet$ is separable, \cref{def:randomVar,def:strongRandomVar} are equivalent (see
\cite[Definitions 1.10, 1.13, and page 16]{bharucha-reidRandomIntegralEquations} and
\cite[Definition 1.14, and Section 1.1.a]{hytonenAnalysisBanachSpaces2016a}), 
\revision{
and so are the underlying concepts of \textit{measurability} and \textit{strong
measurability}}. We
often, but not always, work with separable spaces, hence we adopt the notion of
strong random variable, even though we may verify measurability using preimages of
Borel sets when $\YSet$ is separable.
\end{remark}

We introduce the Bochner spaces $L^p(\Omega,\calF,\prob; \YSet)$, or simply
$L^p(\Omega,\YSet)$, where $p \in [1,\infty]$ as the equivalence classes
of \textit{strongly} $\YSet$-valued random variables endowed with norms
\[
  \begin{aligned}
  & \| u \|_{L^p(\Omega,\YSet)} = ( \mean \| u \|^p)^{1/p} = 
    \bigg(
    \int_\Omega \| u(\omega) \|^p_\YSet \, d\prob(\omega)
    \bigg)^{1/p}, \qquad p \in [1,\infty), \\
  & \| u \|_{L^\infty(\Omega,\YSet)} = \pesssup_{\omega \in \Omega} \| u(\omega)
  \|_\YSet. \\
  \end{aligned}
\]
\begin{remark}
We emphasise here that this definition of Bochner space, taken from \cite[Definition
1.2.15]{hytonenAnalysisBanachSpaces2016a} requires only strong measurability (not
measurability) and is applicable to inseparable as well as separable Banach spaces
$\YSet$.
\end{remark}
\begin{remark}
  \revision{
  We recall that a statement $S(\omega)$ holds \textit{for almost every $\omega \in
  \Omega$} if there exists a set $\calA \in \Omega$ such that $\prob(\calA) = 0$, and
$S(\omega)$ holds for all $\omega \in \Omega \setminus \calA$. If, in a passage, we
fix $\omega$, it is implied that $\omega \in  \Omega \setminus \mathcal{A}$.}

\end{remark}

\section{Problem with random data \texorpdfstring{on \boldmath$D \times J \times
\Omega$}{}}\label{sec:problemRandomData}

We cast the neural field problems as ODEs with random data, on suitable
function spaces. The first step towards this characterisation is to make a standard
hypothesis on the spatio-temporal domain of the problem. Throughout the paper this
domain is a compact in $\RSet^d \times \RSet$.

\begin{hypothesis}[Spatio-temporal domain]\label{hyp:domain} It holds $(x,t) \in D
  \times J$, where $D\subset \mathbb{R}^d$ is a compact domain with piecewise smooth
  boundary, and $J = [0,T] \subset \RSet$.
\end{hypothesis}

%Potential sources of uncertainty in the neural field model may come from: the
%synaptic kernel $w$, the firing rate $f$, the external input $g$, and the
%initial condition $v$. For the sake of simplicity, we are ignoring uncertainties in
%the cortical domain $D$, as well as uncertainties in the decaying rate of the
%5variable $u$, which is set to $-1$, deterministically. 

We formalise our sources of randomness in a similar fashion to Zhang and
Gunzburger, who studied linear parabolic PDEs~\cite{Zhang.2012}. We therefore
consider the probability spaces 
$(\Omega_w,\calF_w,\prob_w)$,
$(\Omega_f,\calF_f,\prob_f)$,
$(\Omega_g,\calF_g,\prob_g)$,
$(\Omega_{v},\calF_{v},\prob_{v})$ or, in compact form
\[
  (\Omega_\alpha,\calF_\alpha,\prob_\alpha), \qquad \alpha \in \USet = \{ w,f,g,v \}. 
\]

We introduce the mappings
\begin{equation}\label{eq:randomData}
\begin{aligned}
  & w \colon D \times D \times \Omega_w \to \RSet, 
  && f \colon \RSet \times \Omega_f \to \RSet, \\
  & g \colon D\times J \times \Omega_g \to \RSet, 
  && v \colon  D \times \Omega_v \to \RSet,
\end{aligned}
\end{equation}
and we are interested in how uncertainty is propagated by the neural field
model \cref{eq:NFDeterministic}.

We consider two separate cases: one in which $f(u)$ is bounded but nonlinear, and
one in which is linear, namely $f(u) = u$. As mentioned above, that the two problems
require different
assumptions, and a separate treatment.
%: the firing functions $f$ chosen in the
%literature are bounded and, as we shall see, this makes their analysis different from
%the (unbounded) linear case $f(u)=u$. Also, the nonlinear
%firing rate is a random field, whereas the firing rate in the linear case is
%deterministic, and equal to $u$. 
% On the other hand, for the linear problem we have
% stronger results in terms of analyticity properties of the solution. 
In the treatment below, we present the problems in parallel
whenever possible. To streamline the notation, we make use of the binary index
$\LSet$, which takes the value $1$ for the linear problem, and $0$ for the nonlinear
one.

We make the natural assumption that the sources of noise are independent, as follows.

\begin{hypothesis}[Independence] The random fields $w$, $f$, $g$, $v$ are mutually
  independent: the event space $\Omega$, $\sigma$-algebra $\calF$, and probability
  measure $\prob$ are given by
\[
  \Omega = \bigtimes_{\alpha \in \USet} \Omega_\alpha, \qquad 
  \calF = \bigtimes_{\alpha \in \USet} \calF_\alpha, \qquad 
  \prob = \prod_{\alpha \in \USet} \prob_\alpha, \qquad
  \quad  
  \mathbb{U} = 
  \begin{cases}
    \{w,g,v\} & \text{if $\LSet = 1$,} \\[0.5em]
    \{w,f,g,v\} & \text{if $\LSet = 0$.}
  \end{cases}
\]
\end{hypothesis} 

An event $\omega \in \Omega$ is written, using its components, as $\omega = \{
 \omega_\alpha \colon \alpha \in \USet\}$.
We can now define informally the neural field problems with random data: given $w$,
$g$, $v$, and possibly $f$ in \cref{eq:randomData}, we seek for a mapping $u \colon D\times J \times
\Omega \to \RSet$ such that for $\prob$-almost all $\omega \in
\Omega$
\begin{equation}\label{eq:nonlinRNF}
  \begin{aligned}
    &\partial_t u(x,t,\omega)  =  -u(x,t,\omega) + g(x,t,\omega_g) 
     + \int_D w(x,x',\omega_w) f(u(x',t,\omega),\omega_f)\, dx',  \\
    &u(x,0,\omega)  = v(x,\omega_{v}),                   
\end{aligned}
\end{equation}
or, in the linear case,
% Similarly, a linear neural field problem with random data can be formalised by means
% of a different set of indices $ \mathbb{U} =\{w,g,v\}$.
% We seek for a function
% $u \colon J \times D \times \Omega \to \RSet$ such that for $\prob$-almost all $\omega \in
% \Omega$ 
\begin{equation}\label{eq:linRNF}
  \begin{aligned}
    & \partial_t u(x,t,\omega)  =  -u(x,t,\omega) + g(x,t,\omega_g) 
			       + \int_D w(x,x',\omega_w) u(x',t,\omega)\,dx',\\
    & u(x,0,\omega) = v(x,\omega_{v}).                         
\end{aligned}
\end{equation}
We now aim to formalise the concept of solutions to the neural field with random data,
and to establish conditions for the existence and uniqueness of such solutions.

\subsection{Evolution equations in operator form \texorpdfstring{on \boldmath$J \times
\Omega$}{}} We begin by casting 
\cref{eq:linRNF,eq:nonlinRNF} as ODEs on a Banach space $\XSet$ with random data.
%As we discuss below, $\XSet$ is a Banach space of functions defined on $D$. We deal
%concurrently with the space of continuous functions on $D$, and the space of
%square-integrable functions on $D$, as seen below.

\begin{hypothesis}[Phase space]\label{hyp:phaseSpace} The phase space is either
  $\XSet$ = $C(D)$, the space of continuous functions on $D$ endowed with the supremum
  norm $\| \blank \|_\infty$, or $\XSet = L^2(D)$, the Lebesgue space of
  square-integrable functions on $D$, endowed with the standard Lebesgue norm $\|
\blank \|_2$. We will compactly write $\XSet \in \{ C(D), L^2(D)\}$.
\end{hypothesis}

With the view of rewriting \cref{eq:nonlinRNF,eq:nonlinRNF} as a
Cauchy problem in operator form, we interpret the solution $u$ as a mapping on
$\XSet$, that is, we define
\[
  U \colon J \times \Omega \to \XSet, \qquad U(t , \omega) = u(\blank,t,\omega).
\]
To keep the notation under control, we use the same letter
$u$ to designate both the mapping on $D \times J \times \Omega$ to $\RSet$, and the
corresponding one on $J \times \Omega$ to $\XSet$. A similar consideration is valid
for the forcing $g$. The initial condition $v$ is also seen as a mapping on $\Omega$
to $\XSet$.
%mappings for the forcing and initial conditions
% \[
%  G(t , \omega_g) = g(t,\blank,\omega_g), \quad V(\omega_v) = v(\blank,\omega_v),  
% \]

We introduce a few operators, instrumental for the discussion on the Cauchy
problem. Firstly, we need an operator-valued random variable associated to the synaptic kernel, and
whose realisations are linear operators on $\XSet$ to itself, namely
\begin{equation}\label{eq:WOp}
  W(\omega_w)(v) := \int_D w(\blank,x',\omega_w) v(x')\,dx'. 
\end{equation}
Secondly we introduce a Nemytskii operator associated to the firing rate
\begin{equation}\label{eq:Nemytskii}
  F(u,\omega_f)(x) := 
  \begin{cases}
    u(x) & \text{if $\LSet = 1$,} \\
  f(u(x),\omega_f)& \text{if $\LSet = 0$,}
  \end{cases}
\end{equation}
and, thirdly, a mapping for the vector field
\begin{align}
 & N(t,u,\omega_w, \omega_f, \omega_g) := -u +
 W(\omega_w)F(u , \omega_f) + g(t , \omega_g) \label{eq:NDef}.
\end{align}
Note that, with the definitions above, selecting $\LSet = 1$ makes $F(u,\omega_f)$
independent of $\omega_f$, hence deterministic.
The problem of finding a random field $u \colon J \times D \times \Omega \to \RSet$
satisfying \cref{eq:nonlinRNF} or \cref{eq:linRNF} $\prob$-almost surely can now be
formalised in the problem below.

\begin{problem}[Random Neural Field on $J \times \Omega$]\label{prob:NRNF}
Fix $\LSet,w,g,v$, and possibly $f$. Find a random solution $u \colon J \times \Omega
\to \XSet$ to the Random Neural Field equation\footnote{A few considerations on
    the notation used here and henceforth. Firstly, we have indicated
    with a prime differentiation with respect to time, that is $u'(t,\omega) :=
    \partial_t u(t,\omega)$. Secondly, we will sometimes write, with a small abuse of
    notation $N(\blank,\blank , \omega)$ in place of the more cumbersome
    $N(\blank,\blank , \omega_w, \omega_g)$, when $\LSet = 1$, or
$N(\blank,\blank,\omega_w,\omega_f,\omega_g)$, when $\LSet = 0$.
}
\begin{equation}\label{eq:NRNFOp}
  \begin{aligned}
  & u'(t , \omega) = N(t, u(t , \omega) , \omega), \qquad t \in J,\\
  & u(0 , \omega) = v(\omega_v),
  \end{aligned}
\end{equation}
that is, a mapping $u$ satisfying
\begin{remunerate}
  \item\label{prob:C1Rgularity} $\omega \mapsto u(\blank,\omega)$ is a
    strong $C^1(J,\XSet)$-valued
    random variable.
  \item\label{prob:probabilityOne} $\prob( B_0 \cap B_J )=1$, with
    \[
        B_0 = \{ \omega \in \Omega \colon u(0, \omega) = v(\omega_v) \}, 
        \qquad  
        B_J = \{ \omega \in \Omega \colon 
            u'(\blank , \omega) = N(\blank , u(\blank , \omega) , \omega) \text{ on $J$} 
          \}.
    \]
\end{remunerate}
\end{problem}
Note that the first condition on $u$ requires strong measurability, as per
\cref{def:strongRandomVar} or equivalently, in view of the separability of
$C^1(J,\XSet)$ and \cref{rem:strongVsWeak}, measurability as per
\cref{def:randomVar}. We stress that \cref{eq:NRNFOp} is a rewriting of
\cref{eq:nonlinRNF,eq:linRNF}
in operator form but, in order to complete the definition of \cref{prob:NRNF}, it is
necessary to formalise the random
linear and nonlinear operators $W(\omega_w)$ and $N(\blank,\omega)$, and we
henceforth proceed in this direction.

\subsection{Hypotheses on random data realisations}
To make progress on the well-posedness of \cref{prob:NRNF}, we make
some technical assumptions on the random fields \cref{eq:randomData}, and discuss
consequent ancillary results. Some of these assumptions originate from the ones
required for the well-posedness of deterministic neural fields. 

We begin by introducing a Banach space useful to discuss the synaptic kernel. For a
fixed bivariate function $k \colon D \times D \to \RSet$, we let
\[
  \nu(h;k) = \max_{x,z \in D} \max_{\| x-z \|_2 \leq h } 
      \int_D |k(x, x') - k(z, x') |\, dx', \qquad h \in \RSet_{ \geq  0},
\]
and we define
\[
  \WSet(\XSet) := 
  \begin{cases}
    \displaystyle{
    \Big\{ 
      k \in C(D,L^1(D)) \colon 
      \lim_{h \to 0} \nu(h; k) = 0
    \Big\},
  }
    & \text{if $\XSet = C(D)$,} \\[1em] 
    L^2(D \times D), & \text{if $\XSet = L^2(D)$.} \\
  \end{cases}
\]
We henceforfth write $\WSet$ in place of $\WSet(\XSet)$, and endow it with the usual
norm on $C(D,L^1(D))$ or $L^2(D\times D)$, respectively,
\[
  \| k \|_\WSet = 
  \begin{cases}
  \displaystyle{ \max_{x \in D} \int_D |k(x, x')| \, dx', }& \text{if $\XSet = C(D)$,} \\[1em]
  \displaystyle{ \bigg( \int_D\int_D |k(x, x')|^2 \,dx\, dx'\bigg)^{1/2}}, 
                                                          & \text{if $\XSet = L^2(D)$}. \\
  \end{cases}
\]
In passing, we note that the space $\WSet$ is separable because $C(D,L^1(D))$ and
$L^2(D\times D)$ are separable, and subspaces of separable metric spaces are
separable \cite[Problem 16G.1]{willardGeneralTopology2004}. 
We now define a linear mapping $H$ which associates a kernel to the corresponding
integral operator. This mapping is useful in defining the random operator
$W(\omega)$ appearing  in \cref{eq:WOp} as a random variable on a suitable Banach space.

% \begin{proposition}[restate, end]\label{prop:HMapping}
\begin{proposition}[Properties of $H$]\label{prop:HMapping}
  Let $H$ be the linear mapping associating
  a kernel $k$, to the corresponding integral operator, namely
  \begin{equation}\label{eq:HDef}
    H(k)(v) = \int_D k(\blank, x') v(x') \, dx'.
  \end{equation}
  \revision{Then $H \colon \WSet \to K(\XSet) \subset BL(\XSet)$ is continuous, where
  $K(\XSet)$ is the space of compact linear operators on $\XSet$}, and the following
  bound hold
  \begin{equation}\label{eq:TBound}
    \| H(k) \|_{BL(\XSet)} \leq \| k \|_\WSet, \qquad \text{for all $k \in \WSet$.}
  \end{equation}
  In addition, the image of $\WSet$ under $H$, that is, $H(\WSet) = \{ H(k) \in
  K(\XSet) \colon k \in \WSet \}$ is a separable subspace of $BL(\XSet)$.
\end{proposition}
\begin{proof}
  See \hyperref[proof:HMapping]{proof}
  % \cref{proof:HMapping} 
  on page \pageref{proof:HMapping}.
\end{proof}

We can now formulate some hypotheses on the random input data.
\begin{hypothesis}[Random data]\label{hyp:randomData} It holds that:
  \begin{remunerate}
    \item \label{hyp:kernel} 
      $\omega \mapsto w(\blank,\blank,\omega)$ is a strongly $\prob_w$-measurable
      $\WSet$-valued random variable;
    \item \label{hyp:input} 
      $\omega \mapsto g(\blank,\blank,\omega)$ is a strongly $\prob_g$-measurable
      $C^0(J,\XSet)$-valued random variable;
    \item \label{hyp:initial} $\omega \mapsto v(\blank,\omega)$ is a strongly $\prob_v$-measurable
      $\XSet$-valued random variable;
    \item \label{hyp:firingRate} $\omega \mapsto f(\blank,\omega)$ is a strongly $\prob_f$-measurable
      $BC^1(\RSet)$-valued random variable.
      % in addition, the mapping $u \mapsto f(u,\omega)$ is
      % Lipschitz for $\prob_f$-almost all $\omega \in \Omega_f$.
  \end{remunerate}
  
\end{hypothesis}

We note that the hypothesis on the firing rate comes into play only in the nonlinear case. In
principle, differentiability in \cref{hyp:randomData}.\labelcref{hyp:firingRate} can
be weakened, but we will keep it here as it simplifies the analysis and it is met in
most mathematical studies on neural fields.
%Henceforth we shall refer to \crefrange{hyp:domain}{hyp:randomData} as \textit{the general asumptions}.
%Our working
These hypotheses imply that realisations of the random data satisfy
requirements usually met in functional-analytic studies of deterministic neural fields
\cite{faugeras2008absolute,Potthast:2010kb,avitabileProjectionMethodsNeural2023}.
They also
guarantee the existence of certain random variables
$\kappa_w$, $\kappa_g$, $\kappa_v$ and $ \kappa_f$ which will be
useful later. They can be interpreted as the magnitude of realisations of the input
data, measured in the respective function space norm. We summarise results in the
following proposition.
\begin{proposition}[Properties of random data]\label{prop:kappaEst}
  Under \crefrange{hyp:domain}{hyp:randomData} we have %the following hold:
  \begin{remunerate}
    \item \label{prop:WBound} The mapping $\omega \mapsto \kappa_w(\omega) := \| w(\blank,\blank,\omega) \|_\WSet$
      is a strongly $\prob_w$-measurable random variable. Further, the mapping 
      \[
        W \colon \Omega_w \to H(\WSet), \quad \omega \mapsto
        H(w(\blank,\blank,\omega)),
      \]
      with $H$ defined as in \cref{eq:HDef}, is a strongly $\prob_w$-measurable
      $H(\WSet)$-valued random variable, whose realisations $W(\omega)(v) = \int_D w(\blank,x',\omega) v(x')
      dx'$ satisfy, for almost all $\omega \in \Omega_w$
  \[
    \| W(\omega) \|_{BL(\XSet)} \leq  %\| w(\blank,\blank,\omega_w) \|_\WSet
    \kappa_w(\omega).
  \]
\item The mapping $\omega \mapsto \kappa_g(\omega) := \| g(\blank,\blank,\omega)
  \|_{C^0(J,\XSet)}$ is a strongly $\prob_g$-measurable random variable.
  For almost all $\omega \in \Omega_g$ the realisations
  $g(\blank,\omega)$ of the forcing satisfy %\label{eq:gBound}
\[
   \| g(t , \omega)\|_\XSet 
  \leq 
  % \| g(\blank , \omega) \|_{C^0(J,\XSet)}
  % =:
  \kappa_g(\omega), \qquad t \in J.
\]
\item
  The mapping $\omega \mapsto \kappa_v(\omega) := \| v(\blank,\omega) \|_{\XSet}$,
  is a strongly $\prob_v$-measurable random variable. 
\item For almost all $\omega \in \Omega_f$, the mapping $u \mapsto F(u,\omega)$ given
  in \cref{eq:Nemytskii} is on $\XSet$ to itself. Further, if
  $\LSet = 1$ then for any $\omega \in \Omega_f$ it holds
  \[
  | F(u,\omega)(x) | = |u(x)|, \quad
  \| F(u,\omega)\|_\XSet = \| u \|_\XSet, \quad
  x \in D, \quad u \in \XSet,
  \]
  whereas if $\LSet = 0$, then for almost every  $\omega \in \Omega_f$
  \[
  | F(u,\omega)(x) | \leq \kappa_f(\omega), 
  \qquad
  \| F(u,\omega)\|_\XSet \leq \kappa_D \kappa_f(\omega), 
  \quad x \in D, \quad u \in \XSet,
  \]
  with $\kappa_D = \max(1,\sqrt{|D|})$ and
  $\kappa_f(\omega) := \| f(\blank,\omega) \|_\infty$ a strongly
  $\prob_f$-measurable random variable. 
\item \label{prop:NemitskiProperties}
  \revision{
    If $\LSet = 0$, then for any $u \in C^0(J,\XSet)$ the mapping $\omega \mapsto \lambda(\omega)$, with
    $\lambda(\omega)(t) = F(u(t),\omega)$ is a
    strongly $\prob_f$-measurable $C^0(J,\XSet)$-valued random variable. 
    % Further, if
    % $\omega \mapsto u(\blank,\omega)$ is a strongly $\prob$-measurable
    % $C^0(J,\XSet)$-valued random variable, then so is 
    % $\omega \mapsto F(u(\blank,\omega),\omega_f)$.
  }
\item \label{prop:NProperties} For almost all $(\omega_w,\omega_f,\omega_g) \in \Omega_w \times \Omega_f
  \times \Omega_g$, the mapping $(t,u) \mapsto N(t,u,\omega_w,\allowbreak \omega_f,\omega_g)$,
  with $N$ defined in \cref{eq:NDef}, is continuous on $J \times \XSet$ to $\XSet$,
  and satisfies, 
  \begin{equation}\label{eq:LipN}
    \| N(t,u,\omega_w,\omega_f,\omega_g) - N(t,v,\omega_w,\omega_f,\omega_g)
    \|_{\XSet} \leq  \kappa_N(\omega_w,\omega_f) \| u-v \|_{\XSet}
  \end{equation}
  for all $(t,u), (t,v) \in J \times \XSet$, where
  \[
    \kappa_N(\omega_w,\omega_f) := 
    \begin{cases}
      1+ \kappa_w(\omega_w), & \text{if $\LSet = 1$,}\\
      1+ \kappa_D \kappa_w(\omega_w) \kappa_{f'}(\omega_f), & \text{if $\LSet = 0$,}
    \end{cases}
  \]
  and $\kappa_{f'}(\omega_f) := \| \partial_u f(\blank,\omega_f) \|_\infty$ are
  strongly measurable random variables. Further, the following bound holds for all
  $(t,u) \in J \times \XSet$ and almost all $\omega \in \Omega$,
  \begin{equation}\label{eq:BoundN}
    \| N(t,u,\omega_w,\omega_f,\omega_g) \|_{\XSet} \leq
    B_N(\| u \|_{\XSet},\omega_w,\omega_f,\omega_g),
  \end{equation}
  where
  \[
    B_N\bigl(\nu,\omega_w,\omega_f,\omega_g\bigr) = 
    \begin{cases}
      (1+\kappa_w(\omega_w)) \nu + \kappa_g(\omega_g),
      & \text{if $\LSet = 1$,} \\
      \nu + \kappa_w(\omega_w) \kappa_D \kappa_f(\omega_f) + \kappa_g(\omega_g), 
      & \text{if $\LSet = 0$.}
    \end{cases}
  \]
  \end{remunerate}
\end{proposition}
\begin{proof}
  See \hyperref[proof:kappaEst] on page \pageref{proof:kappaEst}.
\end{proof}

\section{Well-posedness and regularity of the solution} 
\label{sec:uniqueness}
We now proceed to discuss the existence and uniqueness of a solution to
\cref{prob:NRNF} (and hence to \cref{eq:nonlinRNF,eq:linRNF}). This result follows
an argument similar to the Picard--Lindel\"of Theorem \cite[Theorem
5.2.4]{atkinson2005theoretical} for the deterministic setup
\cite{faugeras2008absolute,Potthast:2009bd,avitabileProjectionMethodsNeural2023}; in the random case,
however, the fixed point argument must be reworked explicitly as one has to ensure
measurability of the solution.

In addition, the linear and nonlinear problem display different bounds for the solution.
In the linear case we obtain an exponential time growth, derived by combining a
variation of constants formula with Gr\"onwall inequality. In the nonlinear case it
is possible to find bounds that are homogeneous in time: we proceed by majorising the
nonlinear and forcing terms, adapting an argument
proposed for $\XSet = C(D)$ by
Potthast and beim Graben \cite{Potthast:2010kb} (see also \cite[Lemma
7.1.3]{NakamuraInverseMod}) so as to
make it valid also when $\XSet = L^2(D)$. Note also that our hypotheses on the kernel
$w$ differ from the ones in \cite{Potthast:2010kb,NakamuraInverseMod}). 

Since we adapt to the random setup a classical fixed-point argument for the
existence and uniqueness of Cauchy problems on Banach spaces, we begin by collecting
a few properties of a Volterra integral operator instrumental to the proof.

\begin{theorem}[Volterra integral operator]\label{thm:volterra}
  Assume \crefrange{hyp:domain}{hyp:randomData}. The mapping
  $\phi \colon C^0(J,\XSet) \times \Omega \to C^0(J,\XSet)$
  defined by% for every $\omega\in\Omega$ and $u\in C^0(J,\XSet)$ by
  \[
    \phi(u,\omega)(t) = v(\omega_v) + \int_0^t N(s,u(s),\omega_w,\omega_f,\omega_g)\, ds, \quad \quad t\in J,
  \]
  where $N$ is defined in \cref{eq:NDef}, and where the integral is interpreted as an
  $\XSet$-valued Riemann integral, enjoys the following properties:
  \begin{remunerate}
    \item \label{prop:phi_u_cont}  For almost all $\omega \in \Omega$, the map $u
      \mapsto \phi(u,\omega)$ is continuous on $C^0(J,\XSet)$ to $C^r(J,\XSet)$, for
      $r \in \{ 0,1 \}$.
    \item \label{prop:phi_rand_1} If $u \in C^0(J,\XSet)$ then $\omega \mapsto \phi(u,\omega)$
      is a strongly $\prob$-measurable $C^r(J,\XSet)$-valued random variable for $r
      \in \{ 0,1 \}$. 
    \item \label{prop:phi_rand_2} If $u$ is a strongly $\prob$-measurable
      $C^r(J,\XSet)$-valued random variable for some $r \in \{ 0,1 \}$ then
      so is $\omega \mapsto \phi(u(\omega),\omega)$. 
    \item \label{prop:volterra_sol} For almost all $\omega \in \Omega$, a mapping $t
      \mapsto u(t,\omega) \in
      C^0(J,\XSet)$ satisfies \cref{eq:nonlinRNF} if, and only if, it is a fixed
      point of $u \mapsto \phi(u,\omega)$.
  \end{remunerate}
  \begin{proof}[Proof of property 1 in \cref{thm:volterra}]
    Throughout the proof of \cref{thm:volterra} we set $\YSet_r = C^r(J,\XSet)$, for $r \in \{
    0,1 \}$. For almost all $\omega \in \Omega$, the mapping $u \mapsto
    \phi(u,\omega)$ is well-defined on $\YSet_0$ to $\YSet_0$, and on $\YSet_0$ to
    $\YSet_1$. 
    To see this fix 
    % $\omega \in \Omega$, 
    $u \in \YSet_0$, and set
    $y(\omega)(t) = \phi(u,\omega)(t)$. 
    \revision{From \cref{prop:kappaEst}, we deduce the existence of a set
      $\mathcal{A} \in \mathcal{F}$ such that the mapping $t \mapsto
      y(\omega)$ is continuous on $J$ to $\XSet$ for all $\omega \in \Omega \setminus
    \mathcal{A}$, where $\prob(\mathcal{A})= 0$.}
    Using \cref{eq:BoundN} and the definition of
    $\phi$ we have:
    \[
      \| y(\omega) \|_{\YSet_0} \leq \|  v(\omega_v) \|_{\XSet} + TB_N(\| u \|_{\YSet_0},\omega_w,\omega_f,\omega_g) < \infty 
    \]
    Further, the mapping $t \to y(\omega)(t)$ is differentiable on $J$ to $\XSet$, and
    \[
      \| y(\omega) \|_{\YSet_1} = 
      \| y(\omega) \|_{\YSet_0} + \| N(\blank,u,\omega) \|_{\YSet_0} 
      \leq 
      \| y(\omega) \|_{\YSet_0} + B_N(\| u \|_{\YSet_0},\omega_w,\omega_f,\omega_g) < \infty.
    \]
  To prove continuity of the mapping $u \mapsto \phi(u,\omega)$ on $\YSet_0$ to
    $\YSet_0$ and on $\YSet_0$ to $\YSet_1$ for almost all $\omega \in \Omega$,
    consider a sequence $\{ u_n \}_{n\in\NSet}$ converging to $u$ in
    $(\YSet_0, \| \blank \|_{\YSet_0})$. Using the $C^r(J,\XSet)$ norm definition and
    \cref{eq:LipN} we find for almost all $\omega \in \Omega$
    \[
      \begin{aligned}
        \| \phi(u_n,\omega) - \phi(u,\omega) \|_{\YSet_0} 
      & = \sup_{t \in J} \| \phi(u_n,\omega)(t) - \phi(u,\omega)(t) \|_\XSet \\
      % & \leq \sup_{t \in J} \int_{0}^{t}\| N(s,u_n(s),\omega) - N(s,u(s),\omega) \| \,ds \\
      & \leq T \kappa_N(\omega_w,\omega_f) \| u_n - u \|_{\YSet_0}  
    \to 0 \quad \text{ as $n \to \infty$},
      \end{aligned}
    \]
    and
    \[
      \begin{aligned}
        \| \phi(u_n,\omega) - \phi(u,\omega) \|_{\YSet_1} 
      & = \| \phi(u_n,\omega) - \phi(u,\omega) \|_{\YSet_0}
        + \sup_{t \in J} \| N(t,u_n(t),\omega) - N(t,u(t),\omega) \|_\XSet 
      \\
      & \leq  (1+T) \kappa_N(\omega_w,\omega_f) \| u_n - u \|_{\YSet_0}  
    \to 0 \quad \text{ as $n \to \infty$}.
      \end{aligned}
    \]
  \end{proof}
  \begin{proof}[Proof of property 2 in \cref{thm:volterra}]
    \revision{
      We prove the statement for $\LSet = 0$ using the auxiliary mapping
      \begin{equation}\label{eq:PsiAux}
        \begin{aligned}
          \psi \colon  \YSet_0 \times \XSet \times K(\XSet) \times \YSet_0 \times \YSet_0 
                & \to \YSet_0 \\
          (u,v,A,\lambda,\gamma) & \mapsto v + \int_{0}^{t} -u(s) + A \lambda(s) + \gamma(s) \,ds
        \end{aligned}
      \end{equation}
      and the fact that continuous transformations of strongly-measurable
      functions are strongly measurable~\cite[Corollary
      1.13]{vanneervenStochasticEvolutionEquations2008}.
      The mapping $\psi$ is clearly well-defined on
      $\YSet_0 \times \XSet \times K(\XSet) \times \YSet_0 \times \YSet_0$ to
      $\YSet_r$. It is also continuous because if $(\bar u, \bar v, \bar A, \bar
      \lambda,
      \bar \gamma) \to (u, v, A, \lambda, \gamma)$ then
      \[
        \begin{aligned}
        \| \psi(\bar u, \bar v, \bar A, \bar \lambda, \bar \gamma) 
          - \psi(u, v,  A,  \lambda, \gamma) \|_{\YSet_0} 
          &\leq 
          % \begin{aligned}[t]
          %  \| \bar v - v \|_{\XSet_0} 
          % &+ T  \| \bar u - u \|_{\YSet_0}
          % + T \| \bar \gamma - \gamma \|_{\YSet_0} \\
          % & + T \| \bar \lambda \|_{\YSet_0} \| \bar A - A \|_{BL(\XSet)}
          % + T  \| A \|_{BL(\XSet)} \| \bar \lambda - \lambda \|_{\YSet_0} \\
          % \end{aligned}
          \begin{aligned}[t]
           \| \bar v - v \|_{\XSet_0} 
          &+ T \bigl(  \| \bar u - u \|_{\YSet_0}
          + \| \bar \gamma - \gamma \|_{\YSet_0} \\
          &  \| \bar \lambda \|_{\YSet_0} \| \bar A - A \|_{BL(\XSet)}
          + \| A \|_{BL(\XSet)} \| \bar \lambda - \lambda \|_{\YSet_0} \bigr) \\
          \end{aligned}
          \\
          & \to 0, 
        \end{aligned}
      \]
      and
      \[
        \begin{aligned}
        \| \psi(\bar u, \bar v, \bar A, \bar \lambda, \bar \gamma) 
          - \psi(u, v,  A,  \lambda, \gamma) \|_{\YSet_1} 
          &\leq 
          % \begin{aligned}[t]
          %  \| \bar v - v \|_{\XSet_0} 
          % &+ T  \| \bar u - u \|_{\YSet_0}
          % + T \| \bar \gamma - \gamma \|_{\YSet_0} \\
          % & + T \| \bar \lambda \|_{\YSet_0} \| \bar A - A \|_{BL(\XSet)}
          % + T  \| A \|_{BL(\XSet)} \| \bar \lambda - \lambda \|_{\YSet_0} \\
          % \end{aligned}
          \begin{aligned}[t]
           \| \bar v - v \|_{\XSet_0} 
          &+ 2T \bigl(  \| \bar u - u \|_{\YSet_0}
          + \| \bar \gamma - \gamma \|_{\YSet_0} \\
          &  \| \bar \lambda \|_{\YSet_0} \| \bar A - A \|_{BL(\XSet)}
          + \| A \|_{BL(\XSet)} \| \bar \lambda - \lambda \|_{\YSet_0} \bigr) \\
          \end{aligned}
          \\
          & \to 0. 
        \end{aligned}
      \]
    }
    
    \revision{
      We now fix fix $u \in \YSet_0$ and let $z(\omega) = \phi(u,\omega)$. To prove
      the statement we show that $z$ is a strongly $\prob$-measurable
      $\YSet_r$-valued random variable. We have
      \[
        z(\omega) = \psi(u,v(\omega_v), 
        A(\omega_w),
        % H(w(\omega_w)), 
        \lambda(\omega_f),
        % F(u,\omega_f),
        g(\blank,\omega_g)),
        \quad 
        A(\omega_w) = H(w(\omega_w)), 
        \quad 
        \lambda(\omega_f)(t) = F(u(t),\omega_f),
      \]
      where $H$ is the operator defined in \cref{prop:HMapping}, and we reason as
      follows:
    \begin{remunerate}
       \item By \cref{hyp:randomData} $v(\omega_v)$ is strongly $\prob_v$-measurable
         and $\XSet$-valued.
       \item By \cref{hyp:randomData} $w(\omega_w)$ is strongly $\prob_w$-measurable
         and $\WSet$-valued, and by \cref{prop:HMapping} the mapping $H \colon \WSet
         \to K(\XSet)$ is continuous, hence by \cite[Corollary 1.13]{vanneervenStochasticEvolutionEquations2008}
        $A(\omega_w)$ is strongly $\prob_w$-measurable and $K(\XSet)$-valued.
      \item By \cref{prop:kappaEst}.\labelcref{prop:NemitskiProperties}
        $\lambda(\omega_f)$ is strongly $\prob_f$-measurable and $\YSet_0$-valued.
       % \item $\omega_f \mapsto \lambda(\omega_f)$ with $\lambda(\omega_f)(t) =
       %   F(u(t),\omega_f)$ 
       %   \cref{prop:kappaEst}.\labelcref{prop:NemitskiProperties}.
      \item By \cref{hyp:randomData} $g(\omega_g)$ is strongly $\prob_g$ measurable $\YSet_0$-valued.
      \item Recalling $\prob = \prob_v \prob_w \prob_f \prob_g$ and setting 
        $\BSet = \XSet \times K(\XSet) \times \YSet_0 \times \YSet_0$ we conclude that 
        \[
          \rho(\omega) = \bigl( v(\omega_v), A(\omega_w), \lambda(\omega_f),g(\blank,\omega_g)\bigr),
        \]
        is strongly $\prob$-measurable and $\BSet$-valued.
      \item The mapping $z$ is thus composition of the strongly $\prob$-measurable
        $\BSet$-valued random variable $\rho$ and the mapping $\rho \mapsto
        \psi(u,\rho)$, which is continuous on $\BSet$ to $\YSet_r$. By
        \cite[Corollary 1.13]{vanneervenStochasticEvolutionEquations2008} $z$ is a
        strongly $\prob$-measurable and $\YSet_r$-valued random variable.
    \end{remunerate}
    The proof for $\LSet = 0$ is now complete. The proof for the linear case $\LSet =
    1$ follows similar steps, upon setting $z(\omega) = \psi(u,v(\omega_v),
    A(\omega_w), u, g(\blank,\omega_g))$.
  }
  \end{proof}

    \begin{proof}[Proof of Property 3 in \cref{thm:volterra}] Fix $r \in \{ 0,1 \}$,
      and let $u$ be a strongly $\prob$-measurable
      $\YSet_r$-valued random variable, and set $y(\omega)=
      \phi(u(\omega),\omega)$, 
      % \revision{for all $\omega \in \Omega \setminus \mathcal{A}$}. 
      We show that $y$ is a
    $\prob$-measurable $\YSet_r$-valued random variable which, owing to the separability of
    $\YSet_r$, is equivalent to $y$ being a strongly $\prob$-measurable $\YSet_r$-valued random
    variable. We adapt \cite[Theorem 2.14]{bharucha-reidRandomIntegralEquations} to
    the case of a nonlinear random operator.

    Let $\{u_n\}$ be a sequence of $\prob$-simple $\YSet_r$-valued random variables of
    the form $u_n(\omega) = \sum_{m=1}^{s_n} 1_{A_{nm}}(\omega) \xi_{nm}$, with
    $A_{nm} \in \calF$ and $\xi_{nm} \in \YSet_r$ for all $m \in \NSet_{s_n}$ and $n \in
    \NSet$, satisfying $\| u_n(\omega)-u(\omega) \|_{\YSet_r} \to 0$ as $n \to \infty$ for
    almost all $\omega \in
    \Omega$, and let $y_n(\omega) = \phi(u_n(\omega),\omega)$ for all $n \in \NSet$.
    For any $B \in \calB(\YSet_r)$ it holds
    \revision{
    \[
      \begin{aligned}
        \{\omega \colon y_n(\omega) \in B \} 
       & =
        \bigcup_{m =1}^{s_n}
        \{ \omega \colon \phi(\xi_{nm},\omega) \in B \}
        \cap
        \{ \omega \colon u_n(\omega) = \xi_{nm} \} \\
       & =
        \bigcup_{m = 1}^{s_n}
        \{ \omega \colon \phi(\xi_{nm},\omega) \in B \}
        \cap A_{nm}
        \in \calF.
      \end{aligned}
      % \begin{aligned}
      %   \{ \omega \in \Omega \colon y_n(\omega) \in B \} 
      %  & =
      %   \bigcup_{m =1}^{s_n}
      %   \{ \omega \in \Omega \colon \phi(\xi_{nm},\omega) \in B \}
      %   \cap
      %   \{ \omega \in \Omega \colon u_n(\omega) = \xi_{nm} \} \\
      %  & =
      %   \bigcup_{m = 1}^{s_n}
      %   \{ \omega \in \Omega \colon \phi(\xi_{nm},\omega) \in B \}
      %   \cap A_{nm} \in \Omega.
      % \end{aligned}
    \]
    }
    In the last passage, we have used the fact that, by Property 2, $\omega \mapsto
    \phi(\xi_{nm},\omega)$ is a 
    \revision{
    strongly $\prob$-measurable 
    $\YSet_r$-valued 
    (hence measurable and $\YSet_r$-valued) random variable.
    We conclude that for any $n$, $y_{n}(\omega)$ is a
    $\prob$-measurable $\YSet_r$-valued (and hence strongly $\prob$-measurable
    $\YSet_r$-valued) 
  random variable.}
    Using the continuity of $\phi$ proved in Property 1, we obtain
    \[
      \lim_{n \to \infty}y_n(\omega) = \lim_{n \to \infty}\phi(u_n(\omega),\omega) =
      \phi(u(\omega),\omega) = y(\omega),
      \qquad
      \text{for almost all $\omega \in \Omega$.}
    \]
    We conclude that $y$ is the $\prob$-almost everywhere limit of a sequence of strongly
    $\prob$-measurable functions, hence by \cite[Definition
    1.1.14]{hytonenAnalysisBanachSpaces2016a} $y$ is strongly $\prob$-measurable.
  \end{proof}

  \begin{proof}[Proof of Property 4 in \cref{thm:volterra}] For almost all $\omega \in \Omega$, the
      operator $u \mapsto \phi(u,\omega)$ satisfies the hypotheses of \cite[Lemma
      2.14 (see also Remark 2.16)]{vanneervenFunctionalAnalysis2022}, and the
      statement follows directly from this result.
  \end{proof}
\end{theorem}

We now proceed to study the well-posedness of \cref{prob:NRNF}. For almost all
$\omega \in \Omega$, the hypotheses of \cite[Theorem 5.2.4]{atkinson2005theoretical}
are satisfied, hence the existence and
uniqueness of a solution $t \mapsto u(t , \omega)$ to \cref{eq:NRNFOp} in $C^1(J,\XSet)$ is
guaranteed $\prob$-almost surely. This path-wise result, however, is not sufficient
to conclude well-posedness in the sense specified by \cref{prob:NRNF}, which also
requires measurability of the random field $u$. We circumvent this problem by
changing intrusively the Picard--Lindel\"of fixed point argument appearing in
deterministic problems, along the lines of
what is done, for instance, in \cite[Theorem 6.7]{bharucha-reidRandomIntegralEquations}. We do this in the theorem
below, where we also present estimates specific to the linear and nonlinear cases.

\begin{theorem}[Solution to neural field problem with random data]\label{thm:existenceRealisationLNF}
  Under \crefrange{hyp:domain}{hyp:randomData}, there exists
  a $\prob$-almost unique strongly
  $\prob$-measurable random variable $u$ with values in $C^r(J,\XSet)$, for $r \in
  \{ 0,1 \}$, solving \cref{eq:NRNFOp} $\prob$-almost
  surely. In addition, there exist
  strong random variables $M$, $M_0$, and $M_1$, such that $\prob$-almost surely
   \begin{align}
    & \| u(t , \omega) \|_\XSet \leq 
    M(\omega), \qquad t \in J \label{eq:UBound}\\
    & \| u(\blank , \omega ) \|_{C^0(J,\XSet)} \leq 
    M_0(\omega),\label{eq:UC0Bound}\\
    %e^{\kappa_wT} ( \kappa_v + \kappa_gT), 
    & \| u(\blank , \omega) \|_{C^1(J,\XSet)} \leq 
    M_1(\omega). \label{eq:UC1Bound}
   \end{align}
  In the linear model, $\LSet =1$, the variables $M$, $M_0$, $M_1$ depend on $t$
  or $T$,
  \[
    \begin{aligned}
      & M(\omega) = 
	\big( \kappa_v(\omega_v)
	+ \kappa_g(\omega_g)t\big)\exp\big(\kappa_w(\omega_w)t\big),  \\
      & M_0(\omega) = 
	\big( \kappa_v(\omega_v)
	+ \kappa_g(\omega_g)T\big)\exp\big(\kappa_w(\omega_w)T\big),  \\
      & M_1(\omega) =
      \kappa_g(\omega_g) + (2+\kappa_w(\omega_w)) M_0(\omega),
    \end{aligned}
  \]
  whereas in the nonlinear model, $\LSet =0$, they are time independent,
  \[
    \begin{aligned}
     & M(\omega) = M_0(\omega), \\
     & M_0(\omega) = 2\max\big[ \kappa_v(\omega_v) \,, \, \kappa_g(\omega_g) + 
     \kappa_D \kappa_w(\omega_w)\kappa_f(\omega_f)\big]\\
     & M_1(\omega) = 2M_0(\omega) +   \kappa_g(\omega_g) + \kappa_D 
       \kappa_w(\omega_w)\kappa_f(\omega_f) , 
    \end{aligned}
  \]
   with $\kappa_D$, $\kappa_w$, $\kappa_f$ $\kappa_g$, and
 $\kappa_v$ defined as in \cref{prop:kappaEst}.
\end{theorem}

\begin{proof} 
  Fix $r \in \{ 0,1 \}$, let $\YSet_r = C^r(J,\XSet)$, and construct the sequence 
  \begin{equation}\label{eq:ynDef}
    y_0(\omega) = v(\omega), \qquad y_{n+1}(\omega) = \phi(y_{n}(\omega),\omega),
    \qquad n  \geq 0,
  \end{equation}
  where the first term is understood as an equality in $\YSet_r$, hence $y_0(\omega)(t)
  = v(\omega)$ for all $(t,\omega) \in J \times \Omega$. We subdivide the proof in
  several steps.

  \textit{Step 1: $ \omega \mapsto y_n(\omega)$ is a strongly $\prob$-measurable
    $\YSet_r$-valued random variable for all $n$.} This is provable by induction. Once
    the base case is proven, 
    % namely if $v(\omega_v)$ 
    % is $\mathbb{P}_v$-measurable then it is also measurable with respect to the product measure $\mathbb{P}$. 
    the induction step follows from the Property 3 of \cref{thm:volterra}.

    By \cref{hyp:randomData}.\labelcref{hyp:initial} $v(\omega)$ is the pointwise
    limit in $\XSet$ of a sequence $\{ v_m \}_m$ of the form
    \[
      v_m(\omega_v)(x) = \sum_{i=1}^{N_m} 1_{A^v_{mi}}(\omega_v) \xi^v_{mi},
      \qquad
      A^v_{mi} \in \calF_v,
      \qquad
      \xi^v_{mi} \in \XSet.
    \]
    We construct the $\prob$-simple $\YSet_r$-valued sequence $\{y_{0m}\}_m$ given by
    \[
      y_{0m}(\omega)(t) = \sum_{i=1}^{N_m} 1_{A_{mi}}(\omega_v) \xi_{mi}(t)
    \]
    with
    \[
       A_{mi} = \Omega_w \times \Omega_f \times \Omega_g \times A^v_{mi} \in \Omega,
       \qquad
       \xi_{mi}(t) \equiv \xi^v_{mi}, \qquad \xi_{mi} \in \YSet_r,
    \]
    and since for almost all $\omega \in \Omega$
  \[
    \| y_{0m}(\omega)- y_0(\omega) \|_{\YSet_r} = \| v_m(\omega_v) - v(\omega_v)\|_\XSet \to 0
    \qquad
    \text{as $m \to \infty$,}
  \]
  we conclude that $y_0(\omega)$ is $\prob$-almost everywhere the pointwise limit of
  a sequence of $\prob$-simple $\YSet_r$-valued functions, and hence is strongly
  $\prob$-measurable and $\YSet_r$-valued \cite[Definition 1.1.14]{hytonenAnalysisBanachSpaces2016a}.
  \Cref{eq:ynDef}, Property \labelcref{prop:phi_rand_2} of \cref{thm:volterra}, and
  mathematical induction ensure that $y_n(\omega)$ is a strongly $\prob$-measurable
  $\YSet_r$-valued random variable for any $n \in \NSet$.

  \textit{Step 2: existence of a fixed point $u(\omega)$ of $\phi(\blank,\omega)$,
    where $\phi \colon \YSet_0 \times \Omega \to \YSet_0$.}
  Using Property \labelcref{prop:NProperties} of \cref{prop:kappaEst} we estimate
  \[
    \| N(t,v(\omega_v),\omega_w,\omega_f,\omega_g) \|_{\XSet} 
    \leq 
    B_N(||v(\omega_v)||_{\XSet},\omega_w,\omega_f,\omega_g) =: B(\omega).
  \]
  Hence for all $t \in J$ and almost all $\omega \in \Omega$, always by Property \labelcref{prop:NProperties}, it holds
  \begin{align*}
    \| y_{1}(\omega)(t) - y_0(\omega)(t)\|_\XSet 
    & \leq  \int_{0}^{t} \| N(s,v(\omega_v),\omega_{w},\omega_f,\omega_g) \|_{\XSet}\,d s \leq tB(\omega), \\
    \| y_{2}(\omega)(t) - y_1(\omega)(t)\|_\XSet 
    & \leq \begin{aligned}[t]
    \int_{0}^{t} \big\| 
        N(s,y_{1}(\omega)(s),&\omega_w,\omega_f,\omega_g) \\
                             &- N(s,y_{0}(\omega)(s),\omega_w,\omega_f,\omega_g) \big\|_\XSet   \,d s 
    \end{aligned} \\
     & \leq \kappa_N(\omega) B(\omega) \int_{0}^{t}s \,ds = \frac{t^2}{2} \kappa_N(\omega) B(\omega),\\
  \end{align*}
  and by induction we find
  \[
    \| y_{n}(\omega) - y_{n-1}(\omega) \|_{\YSet_0} \leq
    B(\omega) \frac{ \kappa_N(\omega)^{n-1} T^{n}}{n!} \to
    0 \qquad \text{as $n \to \infty$,}
  \]
  hence $\{y_n(\omega)\}_n$ converges $\prob$-almost surely to a limit $u(\omega) \in
  \YSet_0$. Using the continuity of $\phi$, established in Property
  \labelcref{prop:phi_u_cont}
  of \cref{thm:volterra}, we have
  \[
    u(\omega) = \lim_{n \to \infty}y_{n+1}(\omega) = \lim_{n \to
    \infty}\phi(y_n(\omega),\omega)
    = \phi(u(\omega),\omega),
    \qquad
    \text{for almost all $\omega \in \Omega$.}
  \]
  Hence $u(\omega)$ is a fixed point of $\phi(\blank,\omega)$ for almost all $\omega
  \in \Omega$.

  \textit{Step 3: $\prob$-almost sure uniqueness of $u(\omega)$.} If $u(\omega),
  z(\omega)$ satisfy $u(\omega) = \phi(u(\omega),\omega)$ and $z(\omega) =
  \phi(z(\omega),\omega)$, respectively, then for almost all $\omega \in \Omega$ it
  holds
  \[
    \| u(\omega)(t) - z(\omega)(t) \|_\XSet \leq  \kappa_N(\omega) \int_0^t \|
    u(\omega)(s) - z(\omega)(s) \|_\XSet\, ds.
  \]
  Using Gronwall's inequality (applying \cite[Chapter 2, Lemma
  6.1]{amannOrdinaryDifferentialEquations2011} with $a(t) \equiv 0$) gives
  \[
    \| u(\omega)(t) - z(\omega)(t) \|_\XSet = 0, \qquad t \in J, \qquad
    \| u(\omega) - z(\omega) \|_{\YSet} = 0, \qquad \text{$\prob$-almost
    surely}.
  \]

  \textit{Step 4: $u$ is a strongly $\prob$-measurable $C^r(J,\XSet)$-valued
  random variable}. Assume $r = 0$.
  We have proved that $u \colon \Omega \to \YSet_0$ is the $\prob$-almost everywhere
  limit of a sequence of strongly $\prob$-measurable functions $y_n \colon \Omega \to
  \YSet_0$, hence \cite[Corollary 1.1.23]{hytonenAnalysisBanachSpaces2016a} implies
  that $u$ is a strongly $\prob$-measurable $\YSet_0$-valued random variable. 
  
  Now assume $r =1$. We have proved that $y_n \colon \Omega \to \YSet_1$ are strongly
  $\prob$-measurable functions. Owing to Property \labelcref{prop:volterra_sol} of
  \cref{thm:volterra}, the mapping $t \mapsto u(\omega)(t)$ is in $C^1(J,\XSet)$ for
  almost all $\omega \in \Omega$. Using \cref{eq:LipN} we estimate
  \[
    \begin{aligned}
    \| y_n(\omega)-u(\omega) \|_{\YSet_1} 
     & = \| y_n(\omega)-u(\omega) \|_{\YSet_0} 
       + \| N(\blank,y_{n-1}(\omega), \omega) - N(\blank, u(\omega),\omega) \|_{\YSet_0} \\
     & \leq  
     \| y_n(\omega)-u(\omega) \|_{\YSet_0} +
     \kappa_N(\omega_w,\omega_f)
       \| y_{n-1}(\omega)-u(\omega) \|_{\YSet_0}  \\
     & \to 0 \qquad \text{as $n \to \infty$},
    \end{aligned}
  \]
  hence $u \colon \Omega \to \YSet_1$ is the $\prob$-almost everywhere
  limit of a sequence of strongly $\prob$-measurable functions $y_n \colon \Omega \to
  \YSet_1$, and \cite[Corollary 1.1.23]{hytonenAnalysisBanachSpaces2016a} implies
  that $u$ is a strongly $\prob$-measurable $\YSet_1$-valued random variable. 

  \textit{Step 5: bounds for $\LSet = 1$.} A bound on $u$ can be found by bounding
  $\|  \phi(u,\blank) \|$ and using \cref{eq:BoundN}. Here we derive a sharper bound
  by using variation of constants, which leads to
  \[
    u(t , \omega) = e^{-t} v(\omega_v) + \int_{0}^t e^{-(t-s)} \big( 
    W(\omega_w)u(s , \omega) + g(s , \omega_g) \big) \, ds.
  \]
  hence for all $t \in J$ and almost all $\omega \in \Omega$
  \[
    \begin{aligned}
      \| u(t , \omega) \|_\XSet 
      & \leq \| v(\omega_v) \|_\XSet + \| g (\blank
	, \omega_g) \|_{\YSet_0} t + \| W(\omega_w) \|_{BL(\XSet)} \int_{0}^t \|
	u(s , \omega) \|_{\XSet} \, ds \\
      & \leq \kappa_v(\omega_v) + \kappa_g(\omega_g)t + \kappa_w(\omega_w)
      \int_{0}^t \| u(s , \omega) \|_{\XSet} \, ds,
    \end{aligned}
  \]
  and Gr\"onwall's Lemma gives $\prob$-almost surely the bounds
  \cref{eq:UBound,eq:UC0Bound}:
  \[
  \begin{aligned}
    & \| u(t , \omega) \|_\XSet \leq 
    e^{\kappa_w(\omega_w)t} ( \kappa_v(\omega_v) +
     \kappa_g(\omega_g)t)) =: M(\omega), & t \in J, \\
   & \| u(\blank , \omega) \|_{\YSet_0} \leq 
    e^{\kappa_w(\omega_w)T} ( \kappa_v(\omega_v) + \kappa_g(\omega_g)T)=:M_0(\omega). \\
  \end{aligned}
\]
To find the bound in the $C^1(J,\XSet)$ norm we use \cref{eq:BoundN} and derive
(omitting $\omega$)
\[
    \| u \|_{\YSet_1} 
   = \| u \|_{\YSet_0} +  \| N(\blank,u,\blank) \|_{\YSet_0} 
   \leq \| u \|_{\YSet_0} +  B_N(\| u \|_{\YSet_0},\blank) 
    \leq \kappa_g + (2+k_w) M_{0} =: M_1
\]
\textit{Step 6: bounds for $\LSet = 0$.} We majorise the nonlinear and forcing terms, which we collect in
\[
  K(x,t,\omega) := g(x,t,\omega_g) + \int_D \! \! w(x,y,\omega_w)
  f(u(t,y,\omega),\omega_f)dy.
\] 
For almost all $(\omega_w, \omega_f,\omega_g)$, it holds
\[
  \| K(\blank,t,\omega)\|_\XSet \leq \kappa_g(\omega_g) + \kappa_D
  \kappa_w(\omega_w)\kappa_f(\omega_f) =: \kappa(\omega),
  \qquad t \in (0,T],
\] 
hence, using integrating factors
\[
  u(x,t,\omega) - e^{-t}v(x,\omega_v) = \int_{0}^t e^{-(t-s)}K(x,s,\omega)\, ds
,\] 
and taking norms
\[
\| u(t,\omega) - e^{-t}v(\omega_v) \|_\XSet \leq \kappa(\omega) \int_{0}^t e^{-(t-s)}\, ds
.\] 
Using the inverse triangle inequality to bound the left-hand side from below, and
integrating on the right-hand side
\[
  \left| 
  \| u(\blank,t,\omega) \|_\XSet - e^{-t} \| v(\blank,\omega_v)\|_\XSet
  \right| 
  \leq \kappa(\omega)( 1 - e^{-t}),
\]
hence
\[
\| u(\blank,t,\omega) \|_\XSet \leq \kappa_v(\omega_v) + \kappa(\omega)( 1 - e^{-t}),
\] 
which implies
\[
\| u(\blank,t,\omega) \|_\XSet \leq 2 \max(\kappa_v(\omega_v),\kappa(\omega))=:M_0(\omega), \qquad t
\in J, \qquad
\text{$\prob$-almost surely,}
\]
that is, \cref{eq:UBound} for $\LSet=0$. The estimate \cref{eq:UC0Bound} follows as
the bound above is homogeneous in $t$. Finally, the bound \cref{eq:UC1Bound} can
be found by estimating
\[
  \begin{aligned}
    \| u(\blank, \omega) \|_{\YSet_1} 
      & = \| u(\blank, \omega) \|_{\YSet_0} +
	      \| u'(\blank, \omega) \|_{\YSet_0} \\
      & \leq 2 \| u(\blank, \omega) \|_{\YSet_0} +\| K(\blank,\blank, \omega)
      \|_{\YSet_0} \\
      & \leq 2 M_0(\omega) + \kappa(\omega).
  \end{aligned}
\]

\end{proof}

\Cref{thm:existenceRealisationLNF} is a step towards characterising solutions to
\cref{prob:NRNF}, in the sense that it provides bounds on realisations of solutions
to \cref{eq:NRNFOp}. In passing, we note that we refer to $u$ as being \textit{the
unique solution} to the problem, even though, strictly speaking $u$ is unique
$\prob$-almost surely. Ultimately, we wish to study the regularity of $u$ as a
$C^0(J,\XSet)$- or $C^1(J,\XSet)$-valued random variable, starting from suitable
hypotheses on $w$, $f$, $g$, and $v$.

To accomplish this task, we must gain control on the random variables in the
bounds \crefrange{eq:UBound}{eq:UC1Bound}, which combine the
variables
$\kappa_\alpha(\omega_\alpha)$, $\alpha \in \USet$.
The random variables $\kappa_g$ and $\kappa_v$ are norms of realisations of
$g$ and $v$ on $C^0(J,\XSet)$ and $\XSet$, respectively, hence we can control them
directly by
demanding that $g$
and $v$ live in an appropriate Banach-valued function space of random functions. On
the other hand, controlling $\kappa_w(\omega)$ does not
necessarily imply controlling $\exp(\kappa_w(\omega)T)$ (and similar for other
products with exponentials), hence further scrutiny is required for the linear model,
as we discuss now.

\subsection{Considerations and further hypotheses on the synaptic
kernel}\label{sec:boundedKernels} 
% \da[inline]{Review this section in light of the split between paper 1 and paper 2. I think Lemma 5.3 is useful
% only in the analyticity paper, and also the distinction into strong and weak
% hypotheses. If so, one could phrase Thm 5.9 using only hypothesis 5.7, and bring the
% simpler hypothesis in only in paper 2}

To address the $L^p$-regularity of the neural field solution in the nonlinear
case, it will be sufficient to demand $L^p$-regularity of the mapping $\omega_w \mapsto
w(\blank,\blank,\omega_w)$, and to use the following result:
\begin{lemma} Assume the hypotheses of \cref{prop:kappaEst}. If $w \in
  L^p(\Omega_w,\WSet)$ then $W \in L^p(\Omega_w,H(\WSet))$.
\end{lemma}
\begin{proof}
  The measurability of the mapping $W$ is established in
  \cref{prop:kappaEst}.\labelcref{prop:WBound}. In addition, from the bound in 
  \cref{prop:kappaEst}.\labelcref{prop:WBound} we have
  \[
    \| W \|^p_{L^p(\Omega_w,H(\WSet))} = 
    \int_{\Omega_w} \| W(\omega_w) \|^p_{BL(\XSet)} \,d \prob_w(\omega_w) 
    \leq  
    \| w \|^p_{L^p(\Omega_w,\WSet)} < \infty.
  \]
\end{proof}

% We now introduce an assumption necessary to prove regularity results in the linear
% model. At a later stage we will present conditions on the kernel $w$ that guarantee
% such hypothesis is satisfied.\da{Say where. A stub of this content is written at the
% end of the manuscript now}
% In the upcoming sections we will study the $L^p$-regularity of solutions to the
% neural fields with random data. For nonlinear neural fields we will work
% under the \crefrange{hyp:domain}{hyp:randomData}; to prove $L^p$-regularity of the
% neural field solution, 

Linear neural fields with random kernels, on the other hand, require further
attention. We introduce a strong regularity assumption, namely that the random field
$w$ is almost-surely bounded in the variable $\omega_w$. 
\begin{hypothesis}[Boundedness of the synaptic kernel in
  $\omega_w$]\label{hyp:wBounded} The synaptic kernel $w$ is in
  $L^\infty(\Omega_w,\WSet)$.
\end{hypothesis}

In addition to being necessary in contexts where analyticity of the solution $u$ is
required (see \cite{avitabile2024Stochastic} for an example in the context of
stochastic collocation schemes),  \cref{hyp:wBounded} makes it easier to check also
certain hypotheses for the existence and well-posedness of linear neural field
problems with random data, albeit it is not strictly necessary in that context. We
shall present a theory based on this strong hypothesis, and signpost that results can
be obtained under weaker assumptions whenever possible.

% Firstly, we note that, under \cref{hyp:wBounded}, one can bound
\revision{Under \cref{hyp:wBounded}, one can bound
$\kappa_w(\omega_w)$
homogeneously in $\omega_w$. If $w \in L^\infty(\Omega_w,\WSet)$, then for almost
every $\omega_w \in \Omega_w$
\[
  \kappa_w(\omega_w) = \| w(\blank,\blank,\omega_w) \|_\WSet \leq \esssup_{ \omega_w
  \in \Omega_w} \| w(\blank,\blank,\omega_w) \|_\WSet = \| w
  \|_{L^\infty(\Omega_w,\WSet)}.
\]
}
% therefore we have the following result:
% \begin{lemma}\label{lem:kappaWConstant}
%   Under \cref{hyp:wBounded} it holds $\kappa_w(\omega_w) \leq \| w
% \|_{L^\infty(\Omega_w,\WSet)}$.
% \end{lemma}

One way to ensure that the random data for the kernel satisfies both
\cref{hyp:randomData}.\labelcref{hyp:kernel} and \cref{hyp:wBounded} 
is to demand that the kernel be
bounded in the spatial variables $x$, $x'$, as well in the stochastic variable
$\omega_w$:
\begin{lemma}\label{hyp:wUniform}
  If $\XSet = C(D)$ and $w \in L^\infty(\Omega_w,C(D \times D))$, or if $\XSet =
  L^2(D)$ and $w \in L^\infty(\Omega_w,L^\infty(D \times D))$, then
  both \cref{hyp:randomData}.\labelcref{hyp:kernel} and \cref{hyp:wBounded} hold.
\end{lemma}
The considerations above provide a direct way to ensure that the strong \cref{hyp:wBounded}
is verified in applications, at the expense of ruling out unbounded random kernels in
linear problems.
% As specified above, these restrictions
% apply only in linear models whereas the
% nonlinear problems are not affected by such restriction. 
%
The strong \cref{hyp:wBounded} can be relaxed in some of the results
presented in this section. We now introduce a weaker hypothesis for linear problems,
harder to verify in applications but sufficient to prove $L^p$-regularity of
solutions.

\begin{hypothesis}[Exponential of kernel norms]\label{hyp:exponential} For $\LSet =
  1$ it holds:
  % \cref{prop:kappaEst}. The kernel $w$ is such that, for some $p\in\NSet$,
  \begin{remunerate}
    \item  For any $t \in \RSet_{\geq 0}$ the random variable $\omega \mapsto
      \exp(\kappa_w(\omega) t)$ is in $L^p(\Omega_w)$.
      \label{hyp:wExp}
    \item  For any $t \in \RSet_{ \geq 0}$, the random variable $\omega \mapsto
      \kappa_w(\omega) \exp(\kappa_w(\omega) t)$ is in $L^p(\Omega_w)$.
      \label{hyp:wTimesWExp}
  \end{remunerate}
\end{hypothesis}
\begin{lemma} If \cref{hyp:wBounded} holds, then so does \cref{hyp:exponential}.
\end{lemma}

\subsection{$L^p$-regularity of the solution}
We now return to studying the regularity of the solution in the linear and nonlinear
neural field with random data, with the following result.
% \da[inline]{Strictly speaking we don't need \cref{hyp:randomData}, in
%   \cref{thm:LpRegularity} because
% requiring $g \in L^p(\Omega_g, C^0(J, \XSet))$, and similar imply that
% \cref{hyp:randomData} holds. Review these hypotheses here and elsewhere, to make them
% 'slimmer'. This will have some repercussions in the statements of pretty much every
% results below: there is a difference between assuming \crefrange{hyp:domain}{hyp:randomData}, and
% assuming the hypotheses of Theorem 4.5, and we should review what we need in
% each circumstance}

\begin{theorem}
  [$L^p$-regularity of the solution with random data]
  \label{thm:LpRegularity}
  \begin{remunerate}
    \item \label{thm:LpRegularityLinear}
      (Linear case): \crefrange{hyp:domain}{hyp:randomData} hold 
      with $\LSet = 1$, and let 
      \revision{$1 \leq p < \infty $}. If $w \in L^\infty(\Omega_w,\WSet)$, $g \in
      L^p(\Omega_g,C^0(J,\XSet))$, and $v \in L^p(\Omega_v,\XSet)$, then the
  solution $u$ to \cref{eq:NRNFOp} is in
  $L^p(\Omega,C^1(J,\XSet))$.
  \item (Nonlinear case):
    Assume the \crefrange{hyp:domain}{hyp:randomData} hold 
    with $\LSet = 0$, and let \revision{$1 \leq p < \infty$}. If $w \in L^p(\Omega_w,\WSet)$, $f \in
    L^p(\Omega_f,BC(\RSet))$ $g \in L^p(\Omega_g,C^0(J,\XSet))$, and $v \in L^p(\Omega_v,\XSet)$, then the solution $u$ to \cref{eq:NRNFOp} is in $L^p(\Omega,C^1(J,\XSet))$. 
\end{remunerate}
\end{theorem}
\begin{proof}
  The proof requires a separate treatment between linear and nonlinear case.
  Assume $\LSet=1$. Owing to the hypotheses on $g$, $v$ and $w$ we have finite
  constants
  \[
  \| \kappa_g \|_{L^p(\Omega_g)} = \| g \|_{L^p(\Omega_g,C^0(J,\XSet))},\quad
  \| \kappa_v \|_{L^p(\Omega_v)} = \| v
  \|_{L^p(\Omega_v,\XSet)}, \quad \kappa_w(\omega_w)\leq \| w\|_{L^\infty(\Omega_w,\WSet)} .
\]
  From \cref{eq:UC0Bound}, the expression for $M_0(\omega)$ for $\LSet=1$ and the
  independence of the random variables $\kappa_w$, $\kappa_g$ and $\kappa_v$ we
  obtain
  \[
    \begin{aligned}
    \| u \|^p_{L^p(\Omega,C^0(J,\XSet))} 
    & \leq \int_{\Omega} e^{p\kappa_w(\omega_w)T} ( \kappa_v(\omega_v) + \kappa_g(\omega_g)T)^p \,
       d\prob_w(\omega_w)d\prob_g(\omega_g)d\prob_v(\omega_v) \\
    & \leq 2^{p-1} e^{Tp \| w\|_{L^\infty(\Omega_w,\WSet)} } \int_{\Omega_g} \int_{\Omega_v}
    (\kappa_v(\omega_v)^p + T^p\kappa_g(\omega_g)^p) \,
       d\prob_g(\omega_g)d\prob_v(\omega_v) \\
    & = 2^{p-1} e^{Tp \| w\|_{L^\infty(\Omega_w,\WSet)}} \Big(  
                \| v \|^p_{L^p(\Omega_v,\XSet)} +T^p \| g \|^p_{L^p(\Omega_g,C^0(J,\XSet))} 
	      \Big) < \infty.
    \end{aligned}
  \]
  From \cref{eq:UC1Bound}, in a similar way we find, omitting the dependence on
  $\omega$, and indicating by $C_p$ a constant dependent on $p$, and whose value may change from passage to passage
  \[
    \begin{aligned}
      \| u \|^p_{L^p(\Omega,C^1(J,\XSet))} 
      & \leq \int_\Omega \Big( \kappa_g + e^{\kappa_w T} (2+\kappa_w)( \kappa_v +
      \kappa_gT) \Big)^p \\
      & \leq C_p \bigg[ \int_{\Omega_g} \kappa^p_g + \int_{\Omega} 
      e^{p\kappa_w T} (2+\kappa_w)^p( \kappa_v + \kappa_gT)^p \bigg]
    \\
    & \leq C_p 
       \bigg[ 
 	\int_{\Omega_g} \kappa^p_g +
 	\int_{\Omega_w} e^{p \kappa_w T} + (\kappa_w e^{\kappa_w T})^p
  	  \int_{\Omega_g} \int_{\Omega_v} \kappa^p_v + (\kappa_gT)^p 
 	\bigg] 
     \\
      &\leq C_p \big[T^p \| \kappa_g \|^p_{L^p} + \big( e^{pT \| w\|_{L^\infty}} + \| w\|_{L^\infty} ^p e^{pT \| w\|_{L^\infty}}\big)
				    \big(\| \kappa_g \|^p_{L^p} + \|
				  \kappa_v\|^p_{L^p} \big) \big] \\
      & < \infty.
   \end{aligned} 
  \]
  Now we pass to the nonlinear case, hence we assume $\LSet = 0$ and use the
  bound \cref{eq:UC0Bound}. Set $\beta = \kappa_g +
 \kappa_D( \kappa_w \kappa_f)$.
  The independence of random variables implies $\beta \in
  L^p(\Omega)$, because
  \[
    \begin{aligned}
    \| \beta  \|^p_{L^p(\Omega)} 
    & \leq 2^{p-1} \int_{\Omega_g} \int_{\Omega_w}
      \int_{\Omega_f} \kappa_g(\omega_g)^p +
      \kappa_D^p \kappa_w(\omega_w)^p \kappa_f(\omega_f)^p d\prob_g d\prob_w d\prob_f \\
    & \leq 2^{p-1} \big( 
        \| \kappa_g \|^p_{L^p(\Omega_g)} 
      + \kappa_D ^p \| \kappa_w \|^p_{L^p(\Omega_w)} 
      \| \kappa_f \|^p_{L^p(\Omega_f)} 
    \big)  \\
    & = 2^{p-1} \big( 
        \| g \|^p_{L^p(\Omega_g,C^0(J,\XSet))} + 
        \kappa_D^p \| w \|^p_{L^p(\Omega_w,BL(\XSet))} 
        \| f \|^p_{L^p(\Omega_f,BC(\RSet))}
	\big) < \infty.
    \end{aligned}
  \]
  Similarly, it holds $M_0 \in L^p(\Omega)$, because
  \[
    \| M_0 \|^p_{L^p(\Omega)} \leq 2^{p-1} \Big( \| v \|^p_{L^p(\Omega_v)} + \| \beta
    \|^p_{L^p(\Omega)} \Big) < \infty.
  \]
  From \cref{eq:UC1Bound} we estimate
  \[
    \begin{aligned}
    \| u \|^p_{L^p(\Omega,C^1(J,\XSet))} 
    & \leq \int_\Omega M_1(\omega)^p d\prob(\omega) = 
      \int_\Omega \! (2M_0(\omega) + \beta(\omega) )^p d\prob(\omega) \\
    & \leq 2^{2p-1} \| M_0 \|^p_{L^p(\Omega)} +
    2^{p-1} \| \beta \|^p_{L^p(\Omega)} < \infty.
    \end{aligned}
  \]
\end{proof}

\begin{remark} As stated in the proof, \cref{thm:LpRegularity}.\labelcref{thm:LpRegularityLinear}, for the
  linear case, relies on the regularity assumption \cref{hyp:wBounded}. It is
  possible to prove a version of this theorem that relies on the milder 
  \cref{hyp:exponential}. In the proof of \cref{thm:LpRegularity}, this is achieved
  by substituting $L^\infty$ norms of $w$ with $L^p$ norms of random variables with
  exponential terms, which are bounded by \cref{hyp:exponential}.
\end{remark}

In what follows, it will be useful to show that the unique solution $u$ to
\cref{eq:NRNFOp} be measurable with respect to some sub $\sigma$-algebras $\calG$ of
$\calF$, as opposed to $\calF$ itself. The result below shows that this is
possible if each random data field is measurable with respect to a sub $\sigma$-algebra
of its original $\sigma$-algebra.

\begin{corollary}[to \cref{thm:LpRegularity}: sub $\sigma$-algebras] \label{cor:subSigma}
  \begin{remunerate}
    \item (Linear case): 
      Assume the hypotheses of \cref{thm:LpRegularity} hold for $\LSet = 1$, and
      let $\calG_\alpha \subset \calF_\alpha$, $\alpha \in \{w,g,v\}$, and $\calG =
  \times_{\alpha} \calG_\alpha \subset \calF$ be sub $\sigma$-algebras\footnote{The
    product of sub $\sigma$-algebras $\calG_\alpha$ is defined by 
  \[
    \calG_w \times \calG_g \times \calG_v = \sigma\big( \{ E_w \times E_g \times E_v \colon 
      E_w \in \calG_w,\; E_g \in \calG_g, \; E_v \in \calG_v\} \big)
  \]
}. If $w$, $g$, $v$ are $\calG_w$-,
  $\calG_g$-, $\calG_v$-measurable, respectively, then $u$ is $\calG$-measurable.
    \item (Nonlinear case):
      Assume the hypotheses of \cref{thm:LpRegularity} hold for $\LSet = 0$, and
      let $\calG_\alpha \subset \calF_\alpha$, $\alpha \in \{w,f,g,v\}$, and $\calG =
  \times_{\alpha} \calG_\alpha \subset \calF$ be sub $\sigma$-algebras. If $w$, $f$, $g$,
  $v$ are $\calG_w$-,$\calG_f$-, $\calG_g$-, $\calG_v$-measurable, respectively, then $u$ is
  $\calG$-measurable. 
  \end{remunerate}
\end{corollary}

\begin{proof}
  See \hyperref[proof:subSigma]{proof} on page \pageref{proof:subSigma}.
\end{proof}

\section{Finite-dimensional noise}
\label{sec:FinDimNoise}
In applications it is often useful to model noise in the input data using a
finite, possibly large number of parameters. This \textit{finite-dimensional noise}
assumption is common when studying PDEs with random
data~\cite{babuskaStochasticCollocationMethod2007,
xiuModelingUncertaintySteady2002,
xiuHighOrderCollocationMethods2005,
nobileSparseGridStochastic2008a,
Zhang.2012,
hoangSparseTensorGalerkin2013a,
Lord:2014ir,
adcockSparsePolynomialApproximation2022} and arises naturally when random
fields are written in terms of a truncated Karhunen-Loeve expansion. For noisy
initial conditions $v$, this would lead to an expression of type
\[
  v(x,\omega) = \mean v(x,\blank) + \sum_{j \in \NSet_q} A_j \psi_j(x) Y_j(\omega),
  \qquad A_j \in \RSet, \quad \psi_j \colon D \to \RSet, \quad Y_j \colon \Omega \to
  \RSet,
\]
where $Y_j$ are iid with zero mean. Nonlinear parametrisations of $v$ in the random
parameters $\{ Y_j \}$ are also
possible~\cite{babuskaStochasticCollocationMethod2007,Lord:2014ir} (see also
\cite{avitabile2024Stochastic} for examples of affine and non-affine parametrisations
in the context of neural fields). It is thus useful to revisit the results of the
Cauchy problem with random data under the finite-dimensional noise assumption, which
we formalise as in \cite[Definition 9.38]{Lord:2014ir}.
\begin{definition}[$m$-dimensional, $p$th-order, $\BSet$-valued noise] 
  \label{def:finDimNoise}
  Let $m,k \in
  \NSet$, and $\BSet$ be a Banach space. Further, let $\{ y_k \}$, $k \in \NSet_m$ be a collection of $m$ independent
  random variables $y_k \colon \Omega \to \Gamma_k \subset \RSet$. A random variable
  $f \in L^p(\Omega,\BSet)$ of the form $f(\blank,y(\omega))$, where
  $y = (y_1,\ldots,y_m) \colon \Omega \to \Gamma = \Gamma_1 \times \dots \times
  \Gamma_m$, is called an
  $m$-dimensional, $p$th-order, $\BSet$-valued noise. We abbreviate this by saying that
  $f \in L^p(\Omega,\BSet)$ is $m$-dimensional noise.
\end{definition}

To explore problems with finite-dimensional random data we work with the following. % assumptions.

\begin{hypothesis}[Finite-dimensional noise random data]\label{hyp:finDimNoise}
  \revision{Let $1 \leq p < \infty $.} There exist random variables $\{ Y_\alpha \}_{\alpha \in
  \USet}$ such that:
  \begin{remunerate}
    \item 
  If $\LSet = 1$ then
    \begin{align*}
      &
      w \in L^\infty (\Omega_w, \WSet), 
      && 
      w(\blank,\blank,\omega_w) = \tilde w(\blank,\blank,Y_w(\omega_w)),
      && 
      Y_w \colon \Omega_w \to \Gamma_w \subset \RSet^{m_w},
      && 
      Y_w \sim \rho_w,
      \\
      &
      g \in L^p(\Omega_f, C^0(J,\XSet)), 
      && 
      g(\blank,\blank,\omega_g) = \tilde g(\blank,\blank,Y_g(\omega_g)),
      && 
      Y_g \colon \Omega_g \to \Gamma_g \subset \RSet^{m_g},
      && 
      Y_g \sim \rho_g,
      \\
      &
      v \in L^p(\Omega_v, \XSet), 
      && 
      v(\blank,\omega_v) = \tilde v(\blank,Y_v(\omega_v)),
      && 
      Y_v \colon \Omega_v \to \Gamma_v \subset \RSet^{m_v},
      && 
      Y_v \sim \rho_v.
    \end{align*}
  \item
  If $\LSet = 0$ then
    \begin{align*}
      &
      w \in L^p (\Omega_w, \WSet), 
      && 
      w(\blank,\blank,\omega_w) = \tilde w(\blank,\blank,Y_w(\omega_w)),
      && 
      Y_w \colon \Omega_w \to \Gamma_w \subset \RSet^{m_w},
      && 
      Y_w \sim \rho_w,
      \\
      &
      f \in L^p(\Omega_f, BC(\RSet)), 
      && 
      f(\blank,\omega_f) = \tilde f(\blank,Y_f(\omega_f)),
      && 
      Y_f \colon \Omega_f \to \Gamma_f \subset \RSet^{m_f},
      && 
      Y_f \sim \rho_f,
      \\
      &
      g \in L^p(\Omega_f, C^0(J,\XSet)), 
      && 
      g(\blank,\blank,\omega_g) = \tilde g(\blank,\blank,Y_g(\omega_g)),
      && 
      Y_g \colon \Omega_g \to \Gamma_g \subset \RSet^{m_g},
      && 
      Y_g \sim \rho_g,
      \\
      &
      v \in L^p(\Omega_v, \XSet), 
      && 
      v(\blank,\omega_v) = \tilde v(\blank,Y_v(\omega_v)),
      && 
      Y_v \colon \Omega_v \to \Gamma_v \subset \RSet^{m_v},
      && 
      Y_v \sim \rho_v.
    \end{align*}
\end{remunerate}
  % \begin{enumerate}
  % \item \label{hyp:finDimNoiseW}
  %   $w \in L^p(\Omega_w,\WSet)$, is $m_w$-dimensional noise of the form
  %   \da[]{must add an $\LSet = 0$, $\LSet =1 $ switch here.}
  % \[
 % w(\blank,\blank,\omega_w) = \tilde w(\blank,\blank,y_w(\omega_w))=: \hat w(\blank,\blank,y(\omega)).
 % \]
% \item \label{hyp:finDimNoiseG}
  % $g \in L^p(\Omega_g,C^0(J,\XSet))$, is $m_g$-dimensional noise of the form
  % \[
  %   g(\blank, \blank,\omega_g) = \tilde g(\blank,\blank,y_g(\omega_g) ) = \hat g(\blank,\blank,y(\omega)).
  % \]
% \item \label{hyp:finDimNoiseV}
  % $v \in L^p(\Omega_v,\XSet)$, is $m_v$-dimensional noise of the form
  % \[
  %   v(\blank,\omega_v) = \tilde v(\blank,y_v(\omega_v) )= \hat v(\blank,y(\omega) ).
  % \]
% \item \label{hyp:finDimNoiseF}
  % $f \in L^p(\Omega_f,BC(\RSet))$, is $m_f$-dimensional noise of the form
  % \[
  %   f(\blank,\omega_f) = \tilde f(\blank,y_f(\omega_f) )= \hat f(\blank,y(\omega) ).
% \]
  % \end{enumerate}
\end{hypothesis}
To make progress in analysing the problem with finite-dimensional noise, we
introduce random fields for the input data
which depend on a single multivariate
random variable $Y$, as opposed to the random variables $\{ Y_\alpha \}$ featuring in
\cref{hyp:finDimNoise}. More precisely we set 
\[
  Y \colon \Omega \to \Gamma = \Gamma_1 \times \ldots \times \Gamma_m \subset
  \RSet^m, \qquad m = \sum_{\alpha \in \USet} m_\alpha,  
  \qquad
  Y \sim \rho = \prod_{\alpha \in \USet} \rho_\alpha,
\]
and consider functions $\hat w$, $\hat f$, $\hat g$, $\hat v$, satisfying $\rho
dy$-almost everywhere
\[
  \begin{aligned}
  \hat w(\blank,\blank,y) := \tilde w(\blank,\blank,y_w),
  \quad 
  \hat f(\blank,y) := \tilde f(\blank,y_w),
  \quad 
  \hat g(\blank,\blank,y) := \tilde g(\blank,\blank,y_g),
  \quad 
  \hat v(\blank,y) := \tilde v(\blank,y_g).
  \end{aligned}
\]
These auxiliary functions, denoted with a hat, are useful in some
contexts, when we want to simplify statements involving $y = \{y_\alpha \colon \alpha
\in \USet\}$, with $\USet = \{w,g,v\}$ (linear case) or
$\USet = \{w,f,g,v\}$ (nonlinear case). Also, we omit hat or tildes, when the context
is clear.

% With the finite-dimensional assumptions, it is possible to rewrite the neural field
% problems with random data as deterministic ones depending on parameters. By the definition of
% finite-dimensional noise, and the hypothesis on independence, we have
% \[
%   y \colon \Omega \to \Gamma = \Gamma_1 \times \ldots \times \Gamma_m \subset
%   \RSet^m, \qquad m = \sum_{\alpha \in \USet} m_\alpha,  
%   \qquad
%   y \sim \rho = \prod_{\alpha \in \USet} \rho_\alpha,
% \]
% and one can make statements on the random fields using the joint distribution $\rho$
% instead of the distributions $\{\rho_\alpha\}$.
%
We expect that a neural field problem with finite-dimensional noise data
admits a finite-dimensional noise solution. This is confirmed by the following
lemma, whose proof adapts \cite[Proposition 4.1]{martinezfrutosOptimalControlPDEs2018} to
the case of multiple noise sources, and to the integro-differential equations under
consideration.

\begin{lemma}[Finite-dimensional noise solution] \label{lem:finDimNoise}
  Under \crefrange{hyp:domain}{hyp:randomData}
  if the random data satisfy the finite-dimensional noise
  \cref{hyp:finDimNoise}, then the solution $u \in L^p(\Omega,C^1(J,\XSet))$ to
  \cref{eq:NRNFOp} is
$m$-dimensional noise of the form
$u(\blank, \blank,\omega) = \tilde u(\blank,\blank,Y(\omega))$
where $m = \sum_\alpha m_\alpha$, and 
$Y(\omega) = (Y_w(\omega_w), Y_g(\omega_g), Y_v(\omega_v))$, if $\LSet= 1$, or
$Y(\omega) = (Y_w(\omega_w), Y_f(\omega_f), Y_g(\omega_g), Y_v(\omega_v))$, if
$\LSet= 0$.
\end{lemma}
\begin{proof}
  See \hyperref{proof:finDimNoise} on page \pageref{proof:finDimNoise}.
\end{proof}

With these premises, under the finite dimensional noise assumption the neural field
problem becomes\footnote{Compare the same statement when one uses only the variables with the tilde
  \[
  \begin{aligned}
  \partial_t \tilde u(t,y(\omega)) =
     &- \tilde u(t,y(\omega)) +  \tilde g(t,y_g(\omega_g)) & \\ 
    & + \tilde W(y_w(\omega_w)) \tilde F(\tilde u(t,y(\omega)), y_f(\omega_f)),
    & t \in (0,T], \text{ $\prob$-a.e. in $\Omega$} \\
    \tilde u(0,y(\omega)) =& \; \tilde v(y_v(\omega_v)).
  & &
  \end{aligned}
\]
Further, when we use the hatted variables below we introduce the space $v \in
L^p_\rho(\Gamma,\XSet)$ in place of $v \in L^p_{\rho_v}(\Gamma_v,\XSet)$ (and all the
variants for each $\alpha \in \USet$).
}
\[
  \begin{aligned}
  \partial_t \tilde u(t,Y(\omega)) =
     &- \tilde u(t,Y(\omega)) +  \hat g(t,Y(\omega)) & \\ 
    & + \hat W(Y(\omega)) \hat F(\tilde u(t,Y(\omega)), Y(\omega)),
    & t \in (0,T], \text{ $\prob$-a.e. in $\Omega$}, \\
    \tilde u(0,Y(\omega)) =& \; \hat v(Y(\omega)).
			     & \text{$\prob$-a.e. in $\Omega$},
  \end{aligned}
\]
which is equivalent to the deterministic problem
\[
  \begin{aligned}
   & \partial_t \tilde u(t,y) =
     - \tilde u(t,y) +  \hat g(t,y) 
    + \hat W(y) \hat F(\tilde u(t,y), y), 
    & t \in (0,T], \text{ $\rho \, dy$-a.e. in $\Gamma$}, \\
    & \tilde u(0,y) = \; \hat v(y).
    & \text{ $\rho \, dy$-a.e. in $\Gamma$}. \\
  \end{aligned}
\]
Dropping tildes and hats, we arrive at the following parametric finite-dimensional problem.

\begin{problem}[Neural field problem with finite-dimensional
  noise]\label{prob:FiniteDimNRNF}
  Fix $\LSet$, $w$, $g$, $v$, and possibly $f$. Given the joint density $\rho =
  \prod_{\alpha \in \USet} \rho_\alpha$ of the multivariate random variable $Y$, find
  $u \colon J \times \Gamma \to \XSet$ such that
\begin{equation}\label{eq:FinDimNRNFOp}
  \begin{aligned}
  & u'(t , y) = N(t, u(t , y) , y), 
  & t \in (0,T], \\
  & u(0 , y) = v(y),
    &%\textrm{$\rho \, dy$-a.e. in $\Gamma$.} \\
  \end{aligned}
  \qquad
  \textrm{$\rho \, dy$-a.e. in $\Gamma$.} \\
\end{equation}
\end{problem}

Further, by introducing suitable function spaces for the finite-dimensional noise variables,
we can study the well-posedness of the problems above. To fix the ideas, let us consider
the random field for the initial condition: for an $m_v$-dimensional noise of the form
$v(\blank,\omega_v) = \tilde v(\blank,Y_v(\omega_v)) = \hat v (\blank,Y(\omega) )$
the following conditions are equivalent
\[
  \begin{aligned}
  & v \in  L^p(\Omega_v,\calF,\prob,\XSet) =: L^p(\Omega_v,\XSet),   \\
  & \tilde v \in L^p(\Gamma_v,\mathcal{B}(\RSet^{m_v}),\rho_v\,dy_v,\XSet)
  =: L^p_{\rho_v}(\Gamma_v,\XSet),  \\
  & \hat v \in L^p(\Gamma, \mathcal{B}(\RSet^{m}), \rho \, dy,\XSet) =: L^p_{\rho}(\Gamma,\XSet),
  \end{aligned}
\]
and analogous considerations are valid for $w$, $f$, $g$, and $u$ in their
respective function spaces. These considerations lead to the following corollary.

\begin{corollary}[to \cref{lem:finDimNoise}: $L^p_\rho$-regularity with
  finite-dimensional noise]\label{cor:uLRho}
Under the hypotheses of \cref{lem:finDimNoise}, it holds $w \in
L^\infty_\rho(\Gamma,\WSet)$ (if $\LSet=1)$ or $w \in
L^p_\rho(\Gamma,\WSet)$ (if $\LSet=0)$, $f \in L^p_\rho(\Gamma,BC(\RSet))$, $g \in
L^p_\rho(\Gamma, C^0(J,\XSet))$, and $v \in
L^p_\rho(\Gamma,\XSet)$. Further, \cref{prob:FiniteDimNRNF} has a unique solution $u \in
L^p_\rho(\Gamma,C^1(J,\XSet))$.
\end{corollary}

\section{Spatially-projected problem with random data}\label{sec:projProblem}
When defining schemes for the approximate solutions of neural field problems with
random data it is required to study properties of solutions to neural
fields in semi-discrete form, that is, after a spatial discretisation has been
applied. We therefore turn our attention to neural fields with random data in
semi-discrete form.

A generic framework for discretising deterministic neural field problems has been
proposed in \cite{avitabileProjectionMethodsNeural2023}, adopting projection
operators \cite{atkinson1997,atkinson2005theoretical}, and we adopt this framework here too. With the view of
discretising space, we introduce a sequence of finite-dimensional
approximating subspaces $\{\XSet_n \colon n \in \NSet\} \subset \XSet$, with $\overline{\cup_{n \in
\NSet} \XSet_n} = \XSet$, $\dim
\XSet_n = s(n) \to \infty$ as $n \to \infty$, and $\XSet_n = \spn\{ \phi_j \colon j
\in \NSet_{s(n)}\}$, where $\{ \phi_j \colon j \in \NSet \}$ is a basis for $\XSet$.
On $\XnSet$ we place the norm $\|\blank \|_\XSet$, and henceforth we consider
subspaces $\left( \XnSet,\|\blank\|_\XSet \right)$.
We introduce a family of bounded projection operators $\{P_n \colon n \in \NSet\}$,
with $P_n \in BL(\XSet,\XSet_n)$. We place on $BL(\XSet,\XnSet)$ the operator norm $\| \blank
\|_{BL(\XSet)}$, for which we will use the unadorned symbol $\| \blank \|$ whenever
possible. For a family of projector operators it holds $\|P_n\| \geq 1$ for all $n
\in \NSet$.

The spatial projectors of interest to us are \textit{interpolatory and orthogonal
projectors}, and an abstract formulation which treats simultaneously both choices is
possible for deterministic neural field equations~\cite{avitabileProjectionMethodsNeural2023}.
Here we study the problem in a similar fashion, but working with neural fields with
random data.
For a function $v \in \XSet$, one defines $P_n$ through the action
\begin{equation}\label{eq:ProjAction}
  (P_n v)(x) = \sum_{j \in \NSet_{s(n)}} V_j \phi_j(x), \qquad x \in D.
\end{equation}
A typical functional setup for schemes with interpolatory projectors in a neural
field involves $\XSet = C(D)$ and leads to $\phi_j=l_j$, with $j\in\ZSet$, where $l_j$ is the $j$th Lagrange interpolation polynomial with nodes $\{ x_j:j\in\ZSet_n\}$ and $V_j = v(x_j)$, whereas in a typical setup for
orthogonal projectors one has $\XSet = L^2(D)$ and $V_j = \langle v,\phi_j
\rangle_\XSet$.

We defer to the literature cited above for examples of concrete choices of the
projectors, and work on the abstract formulation of the problem using
\cref{eq:ProjAction}. We define
spatially-projected schemes to approximate realisations of \cref{eq:NRNFOp}. For
fixed $n \in \NSet$ we consider the problem
\begin{equation}\label{eq:NRNFOp_Pn}
  \begin{aligned}
    u_n'(t,\omega) & = P_n N(t,u_n(t,\omega),\omega), \qquad t \in (0,T],\\
    u_n(0,\omega) & = P_n v(\omega),
  \end{aligned}
\end{equation}
for which we seek for a solution $u_n \in L^p(\Omega,\XSet_n)$, in the
sense given in \cref{prob:NRNF}, approximating the solution $u \in L^p(\Omega,\XSet)$
to \cref{eq:NRNFOp}. 

A useful consideration for unpacking the notation hidden in \cref{eq:NRNFOp_Pn} is to
keep in mind that the projector $P_n$ acts \textit{exclusively on the spatial
  variable $x$, which is not exposed in \cref{eq:NRNFOp_Pn}}. 
For instance, for the initial condition $v \in L^p(\Omega_v,\XSet)$ we write
$P_n v(\omega_v)$, which is consistent with the observation that $\prob_v$-almost
surely $v(\omega_v)\in \XSet$, hence
\[
  (P_n v(\omega_v))(x) = \sum_{i \in \NSet_{s(n)}} V_i(\omega_v) \phi_i(x),
\]
where $V_i(\omega_v) = v(x_i,\omega_v)$ or $V_i(\omega_v) = \langle
v(\blank,\omega_v),\phi_i \rangle_\XSet$.

Secondly, to define $P_n N$ with $N$ given in \Cref{eq:NDef}, the projector $P_n$
must act on realisations of the forcing term $g \in L^p(\Omega_g,C^0(J,\XSet))$, and
so we shall write $P_n g(t,\omega_g)$ to indicate
\[
  (P_n g(t,\omega_g))(x) = \sum_{i \in \NSet_{s(n)}} G_i(t,\omega_g) \phi_i(x),
\]
with the usual considerations for $G_i$.

Thirdly, to project $N$ we must project the action of the integral operator
realisations $W(\omega_w)$, and we formalise this step by composing $P_n$ and
$W(\omega_w)$: 
\begin{equation}\label{eq:PnWDef}
  P_n W(\omega_w) \colon v \mapsto P_n \int_D w (\blank,x',\omega_v) v(x')\,dx'
  = \sum_{i \in \NSet_{s(n)}} \phi_i \int_D W_i(x',\omega_w) v(x') \, dx', 
\end{equation}
where, consistently with our previous notation, for almost all $\omega_w \in
\Omega_w$ it holds $W_i(x',\omega_w) = w(x_i,x',\omega_w)$ for interpolatory
projectors, and $W_i(x',\omega_w) = \langle
w(\blank,x',\omega_w),\phi_i \rangle_\XSet$ for orthogonal projectors.
We use $P_n w$ to
indicate a projection with respect to the variable $x$ only, that is
\[
  (P_n w)(x,x',\omega_w) = P_n w(\blank,x',\omega_w)(x).  
\]

It follows from
\cref{eq:PnWDef} that one can also equivalently interpret $P_n W$ composing the linear
mapping $H$ in \cref{eq:HDef} with a projection of the function $x \mapsto
w(x,\blank,\blank)$, as follows
\begin{equation}\label{eq:HwProj}
    P_n W(\omega_w)  
    = H\Big ( \sum_{i \in \NSet_{s(n)}} W_i(\blank,\omega_ w) \phi_i  \Big )
    = H \big(P_n w(\blank, \blank, \omega_w)\big).
\end{equation}
In addition, for any $u \in \XSet$ and $n \in \NSet_n$ it holds $P_{n}(WF(u)) =
(P_{n}W)F(u)$, which prompts us to use unambiguously $P_n WF(u)$ for an operator on
$\XSet$ to $\XSet_n$. 
 
With these preparations, the vector field of 
% the projected problem
\cref{eq:NRNFOp_Pn} is well defined as
\[
  \begin{aligned}
    P_n N \colon J \times \XnSet \times \Omega & \to \XnSet \\
    (t,u_n,\omega) & \mapsto -u_n + P_n W(\omega_w) F(u_n,\omega_f) 
    + P_n g(t,\omega_g),
  \end{aligned}
\]
and we seek to state for \cref{eq:NRNFOp_Pn} an analogue of
\cref{thm:existenceRealisationLNF}.

We stress that the projected problem \cref{eq:NRNFOp_Pn} econmpasses at the same time
strong and weak
problem formulations in the spatial variable. In
\cite{avitabileProjectionMethodsNeural2023}, it is shown that one realisation of the evolution equation
\cref{eq:NRNFOp_Pn} generates Finite-Element Collocation, Spectral Collocation,
Finite-Element Galerkin, and Spectral Galerkin schemes, obtained by choosing between
interpolatory and orthogonal projectors, and between locally- and
globally-supported basis $\{\phi_i\}$. 

% \subsection{Results on spatial projectors} 
% \da[inline]{See how much of this is needed here, and how much is for paper 2}
% We close the section by discussing a useful result on the spatial projectors and their properties.
%
The boundedness of $P_n$ leads to the following estimates, which are helpful to
transfer bounds on $w$, $g$, and $v$ to bounds on $P_n W$, $P_n g$, and $P_n v$,
respectively.
\begin{proposition}\label{prop:kappaEst_Pn}
  Assume \crefrange{hyp:domain}{hyp:phaseSpace} and
  \cref{hyp:randomData}.\labelcref{hyp:kernel}--\labelcref{hyp:initial} and let
  \[
  \begin{aligned}
    & \kappa_{w,n}(\omega_w) :=
    \| P_n W(\omega_w) \|_{BL(\XSet,\XnSet)}, 
    && \kappa_{w}(\omega_w) :=
    \| w(\omega_w) \|_{\WSet}, 
    \\
    & \kappa_{g,n}(\omega_g) :=
    \| P_n g(\blank,\omega_g) \|_{C^0(J,\XSet)}, 
    && \kappa_{g}(\omega_g) :=
    \| g(\blank,\omega_g) \|_{C^0(J,\XSet)}, 
    \\
    & \kappa_{v,n}(\omega_v) :=
    \| P_n v(\omega_v) \|_{\XSet},
    && \kappa_{v}(\omega_v) :=
    \| v(\omega_v) \|_{\XSet}.
  \end{aligned}
  \]
  
  For any $n \in \NSet$ and  $\alpha \in \{w,g,v\}$ it holds
  \[
    \kappa_{\alpha,n} \leq \| P_n \| \kappa_\alpha \qquad \text{$\prob_\alpha$-almost surely}
  .\] 
\end{proposition}
\begin{proof}
  The bound on $W$ holds because, for any $z \in \XSet$ it holds $\prob_w$-almost surely
    \[
    \begin{aligned}
    \| P_n W(\omega_w) z \|_{\XSet} 
      & \leq \| P_n \|\, \| W(\omega_w) z \|_{\XSet} \\
      & \leq \| P_n \|\, \| W(\omega_w) \|_{BL(\XSet)} \| z \|_{\XSet} \\
      & \leq \| P_n \|\, \kappa_w(\omega_w) \| z \|_{\XSet}.
    \end{aligned}
  \]
  In addition, $\prob_g$-almost surely we have
  \[
    \begin{aligned}
    \| P_n g(\blank,\omega_g) \|_{C^0(J,\XSet)}
      & = \sup_{t \in J} \| P_n g(t,\omega_g) \|_{\XSet} \\
      & \leq \| P_n \|\sup_{t \in J} \|g(t,\omega_g) \|_{\XSet}  \\
      & = \| P_n \| \|g(\blank,\omega_g) \|_{C^0(J,\XSet)}  
      = \|P_n \| \kappa_g(\omega_g),
    \end{aligned}
  \]
  and a similar argument gives the bound for $\|P_n v(\omega_v)\|_\XSet$.
\end{proof}

To study the convergence of a numerical scheme, one must bound the error $u
- u_n$ in a suitably defined norm. This, in turn, requires control on the asymptotic
behaviour of projectors acting on realisations of initial conditions, integral
operators, and external inputs, as $n \to \infty$. This leads to studying, for
instance, the sequence $\{P_n v(\omega_v)\}_n \subset \XSet$ for fixed $\omega_v \in
\Omega_v$, or similar sequences for the integral and forcing operators. 

We do not pursue the study of such convergence here, as this is left to applications
of the present theory, and depends on the scheme employed to approximate the random
field. An example of such study can be found in~\cite{avitabile2024Stochastic}.

For completeness, we state the spatially-projected problem under the finite
dimensional noise assumptions.
\begin{problem}[Spatially-projected problem with finite-dimensional
  noise]\label{prob:FiniteDimNRNF_Pn}
  Fix $\LSet$, $w$, $g$, $v$, and possibly $f$. Given the joint density $\rho =
  \prod_{\alpha \in \USet} \rho_\alpha$ of the multivariate variable $y$, find
  $u_n \colon J \times \Gamma \to \XSet$ such that
\begin{equation}\label{eq:unEquation}
  \begin{aligned}
  & u_n'(t , y) = P_n N(t, u_n(t , y) , y), 
    &  t \in (0,T], \\
  & u_n(0 , y) = P_n v(y),
    &%\textrm{$\rho \, dy$-a.e. in $\Gamma$.} \\
  \end{aligned}
  \qquad \textrm{$\rho \, dy$-a.e. in $\Gamma$.} \\
\end{equation}
\end{problem}

\subsection{Spatially-projected problem \texorpdfstring{on \boldmath$J \times
\Omega$}{with infinite-dimensional noise}}
We can now study the existence of solutions $u_n$ to the projected problem
\cref{eq:NRNFOp_Pn} as a problem on $J \times \Omega$, and relate our findings to the
ones for solutions $u$ to the original problem \cref{eq:NRNFOp}.

\begin{theorem}[Spatially-projected neural field with random data]\label{thm:existenceNF_Pn}
  Under the \crefrange{hyp:domain}{hyp:randomData}, there exists a unique
  strongly measurable $C^1(J,\XSet)$-valued random variable $u_n$ solving \cref{eq:NRNFOp_Pn}
  $\prob$-almost surely, satisfying
   \begin{align}
    & \| u_n(t , \omega) \|_{\XSet} \leq 
    M_n(\omega), \qquad t \in J, \label{eq:UBound_Pn}\\
    & \| u_n(\blank , \omega ) \|_{C^0(J,\XSet)} \leq 
    M_{0,n}(\omega),\label{eq:UC0Bound_Pn}\\
    %e^{\kappa_wT} ( \kappa_v + \kappa_gT), 
    & \| u_n(\blank , \omega) \|_{C^1(J,\XSet)} \leq 
    M_{1,n}(\omega), \label{eq:UC1Bound_Pn}
   \end{align}
   where the random variables $M_{n}$, $M_{0,n}$, and $M_{1,n}$ are 
   derived from $M$, $M_0$, and $M_1$ in \cref{thm:existenceRealisationLNF}, respectively, 
   upon substituting $\kappa_\alpha$ by $\kappa_{\alpha,n}$ with $\alpha \in
   \{w,g,v\}$.

   If, in addition, $P_n z \to z$ for all  $x \in \XSet$, then there exist random
   variables $\bar M$, $\bar M_0$ $\bar M_1$, independent of $n$, such that
   $\prob$-almost surely it holds
   \[
   M_n(\omega) \leq \bar M(\omega),
   \quad
   M_{0,n}(\omega) \leq \bar M_{0}(\omega),
   \quad
   M_{1,n}(\omega) \leq \bar M_{1}(\omega)
   \quad
   n \in \NSet
   .\] 
\end{theorem}
\begin{proof}
  Existence, uniqueness, and measurability of $u_n$ are proved following steps
  identical to the ones in \cref{thm:existenceRealisationLNF}, upon replacing
  the operator $N$ by the projected operator $P_nN \colon J \times \XSet_n \times \Omega \to
  \XSet_n$. In particular, for fixed $r \in \{ 0,1 \}$ and $n \in \NSet$ we set
  $\YSet_{r,n} = C^r(J,\XSet_n)$, construct the sequence 
  \begin{equation}
    y_0(\omega) = P_n v(\omega), \qquad y_{k+1}(\omega) =
    \phi_n(y_{k}(\omega),\omega), \qquad k \geq 0,
  \end{equation}
  where the operator $\phi_n: \YSet_{0,n} \to \YSet_{r,n}$ acts as the Volterra
  operator in \cref{thm:volterra}, with $P_n N$ in place of $N$. Steps 1--4 in
  \cref{thm:existenceRealisationLNF} are readily adapted with minor
  modifications: the sequence $y_k$, replacing $y_n$, is used to prove existence,
  uniqueness, and measurability of the fixed point of $\phi_n$.

  We then turn our attention to the solution bounds, which we prove separately for
  the linear and nonlinear case. Set $\LSet = 1$. Proceeding as in step 5 in the
  proof of \cref{thm:existenceRealisationLNF} we arrive at
  \[
    \begin{aligned}
      \| u_n(t , \omega) \|_{\XSet} 
      & \leq \| P_n v(\omega_v) \|_{\XSet} + \| P_n g (\blank
	, \omega_g) \|_{C^0(J,\XSet)} t  \\
      & \phantom{\leq \| P_n v(\omega_v) \|_{\XnSet}} 
	+ \| P_n W(\omega_w)
	\|_{BL(\XSet,\XSet_n)} \int_{0}^t \|
	u_n(s , \omega) \|_{\XSet} \, ds \\
      & =  
      \kappa_{v,n}(\omega_v) + \kappa_{g,n}(\omega_g)t + \kappa_{w,n}(\omega_w)
      \int_{0}^t \| u(s , \omega) \|_{\XSet} \, ds ,
    \end{aligned}
  \]
  and the Gr\"onwall lemma gives the bound on $\| u_n(t,\omega)\|_\XSet$. The 
  $C^0(J,\XSet)$ and $C^1(J,\XSet)$ bounds follow as in the proof of
  \cref{thm:existenceRealisationLNF}. 
  Now set $\LSet=0$, to address the nonlinear case. We proceed as in
  \cref{thm:existenceRealisationLNF}, with the operator $P_n K$ in place of $K$. Using
  \cref{prop:kappaEst_Pn} we find the following estimate for almost all $(\omega_w,
  \omega_f,\omega_g)$ 
\[
  \| P_n K(\blank,t,\omega)\|_\XSet \leq \kappa_{g,n}(\omega_g) + \kappa_D
  \kappa_{w,n}(\omega_w)\kappa_f(\omega_f) =: \kappa_n(\omega),
  \qquad t \in (0,T].
\] 
Using integrating factors and the inverse triangle inequality we arrive at
  \[
  \left| 
  \| u_n(\blank,t,\omega) \|_\XSet - e^{-t} \| P_n v(\blank,\omega_v)\|_\XSet
  \right| 
  \leq \kappa_n(\omega)( 1 - e^{-t})
\]
hence
\[
  \| u_n(\blank,t,\omega) \|_\XSet \leq 2\max[ k_{v,n}( \omega_v ), \kappa_n(\omega) ]
,\] 
and the result follows as in \cref{thm:existenceRealisationLNF}. We have now
established the bounds \crefrange{eq:UBound_Pn}{eq:UC1Bound_Pn} for both
$\LSet = 0$ and $\LSet= 1$. 

If $P_n z \to z$ for all $z \in \XSet$ then $\sup_{n \in \NSet} \| P_n \| < \infty $
(see~\cite[Theorem 2.4.4]{atkinson2005theoretical}),
and hence by \cref{prop:kappaEst_Pn} the sequences $\{\kappa_{\alpha,n}(\omega_\alpha)\}_n$ are
bounded almost surely, that is, $\kappa_{\alpha,n}(\omega_\alpha) \leq
\bar{\kappa}_{\alpha}(\omega_\alpha)$ for some positive
$\bar{\kappa}_{\alpha}(\omega_\alpha)$. This gives the existence of random
variables $\bar M$, $\bar M_0$, and $\bar M_1$ independent of $n$, bounding $u_n$
from above, obtained by substituting every occurrence of $\kappa_{\alpha,n}$ with
$\bar{\kappa}_\alpha$ in $M_n$, $M_{0,n}$, and $M_{1,n}$, respectively.
\end{proof}

Next, we look at the $L^p$-regularity and measurability with respect to sub
$\sigma$-algebras of the solution to the projected problem.
Notably, if $u$ is $L^p$-regular (or measurable with respect to a sub-$\sigma$
algebra) then
so is $u_n$, for any $n \in \NSet$, that is, the spatial projectors $P_n$ do not
interfere with the regularity of the solution, which is determined by 
the input data. This is a consequence of
\cref{prop:kappaEst_Pn} which gives estimates for the projected variables starting
from the original variables. 

\begin{theorem}[$L^p$-regularity in the spatially-projected problem] 
  \label{thm:LpRegularity_Pn}
  Under the hypotheses of \cref{thm:LpRegularity}, its conclusions hold for the
  unique solution $u_n$ to \cref{eq:NRNFOp_Pn}.
\end{theorem}
\begin{proof}
  If $\LSet=1$, the essential boundedness of $w(\omega_w)$ and \cref{prop:kappaEst_Pn} 
  give the bound $\exp(k_{w,n}(\omega_w) t) \leq \exp( \| w
  \|_{L^\infty(\Omega_w,\WSet)}\| P_n
  \|t) $ for all $t \in \RSet_{ \geq 0}$, and any $p,n \in \NSet$. Since the bound is
  independent of $\omega_w$, it follows that $\exp(k_{n,w} t)
  \in L^p(\Omega_w)$ for all $t \in \RSet_{\geq 0}$ and $p,n \in \NSet$. A similar
  reasoning gives $k_{w,n} \exp(k_{w,n}t) \in L^p(\Omega_w)$. The proof then runs
  identically to the one of \cref{thm:LpRegularity}, for both $\LSet= 1$ and $\LSet=
  0$, with $\kappa_{\alpha}$ replaced by $\kappa_{\alpha,n}$ for all $\alpha \in \{
  w,f,g,v \}$, and is omitted here. 
\end{proof}

% \fc[inline]{Can we state the corollary below as follows: under the hypothesis of
% corollary 4.10 $u_n$ is $\calG$-measurable. or Corollary 4.10 holds for $u_n$.}
% \begin{corollary}[Sub $\sigma$-algebra measurability in the projected problem] 
%   \label{cor:subSigma_Pn}
%   \begin{enumerate}
%     \item  (Linear case): Assume the hypotheses of \cref{thm:LpRegularity_Pn} hold, and
%       let $\calG_\alpha \subset \calF_\alpha$, $\alpha \in \{w,g,v\}$, $\calG =
%       \times_{\alpha} \calG_\alpha \subset \calF$ be sub $\sigma$-algebras. If $w$, $g$, $v$ are $\calG_w$-, $\calG_g$-, $\calG_v$ -measurable, respectively, then $u_n$ is $\calG$-measurable.
%      \item (Nonlinear case): Assume the hypotheses of \cref{thm:LpRegularity_Pn} hold, and
%       let $\calG_\alpha \subset \calF_\alpha$, $\alpha \in \{w,f,g,v\}$, and $\calG =
%   \times_{\alpha} \calG_\alpha \subset \calF$ be sub $\sigma$-algebras. If $w$, $f$, $g$,
%   $v$ are $\calG_w$-,$\calG_f$-, $\calG_g$-, $\calG_v$-measurable, respectively, 
%        then $u_n$ is $\calG$-measurable.
%   \end{enumerate}
% \end{corollary}
\begin{corollary}[Sub $\sigma$-algebra measurability in the projected problem] 
  \label{cor:subSigma_Pn}
  Under the hypotheses of \cref{cor:subSigma}, its conclusions hold for the unique
  solution $u_n$ to \cref{eq:NRNFOp_Pn}.
\end{corollary}
\begin{proof}
    The proof is identical to the one of \cref{cor:subSigma}, and is omitted.
\end{proof}

In this section we have transplanted results obtained for the solution $u$ of
\cref{eq:NRNFOp}, to the solution $u_n$ of \cref{eq:NRNFOp_Pn}. We conclude this
section by discussing properties of $u_n$ under
the finite-dimensional noise assumption.

\begin{lemma}[Finite-dimensional noise in the projected problem]
  \label{lem:finDimNoise_Pn}
  Under the hypotheses of \cref{thm:LpRegularity_Pn},
  if the random inputs satisfy the finite-dimensional noise
  \cref{hyp:finDimNoise}, then the unique solution $u_n \in L^p(\Omega,C^1(J,\XSet_n))$ to
  \cref{eq:NRNFOp} is
$m$-dimensional noise of the form
$u_n(\blank, \blank,\omega) = \tilde u_n(\blank,\blank,Y(\omega))$,
where $m = \sum_\alpha m_\alpha$, and 
$Y(\omega) = (Y_w(\omega_w), Y_g(\omega_g), Y_v(\omega_v))$, if $\LSet= 1$, or
$Y(\omega) = (Y_w(\omega_w), Y_f(\omega_f), Y_g(\omega_g), Y_v(\omega_v))$, if
$\LSet= 0$.
\end{lemma}
\begin{proof}
The proof follows identical steps to \cref{lem:finDimNoise}, where the corresponding 
results on the space-projected problem stated above should be used in place of the ones in \cref{sec:uniqueness}.
\end{proof}

\begin{corollary}[$L^p_\rho$-regularity in the
  projected problem]
  \label{cor:uLRho_Pn}
Under the hypotheses of \cref{lem:finDimNoise}, the unique
solution $u_n$ to \cref{eq:NRNFOp_Pn} is in $L^p_\rho(\Gamma,C^1(J,\XSet_n))$.
\end{corollary}

\section{Conclusions}\label{sec:conclusions} In this paper we have studied neural field
equations as Cauchy problems subject to random data. We have provided theoretical
background and estimates instrumental to prove convergence of numerical schemes, and
to derive their convergence rates. We expect this theory to be employed in schemes
that combine a spatial numerical discretisation of Collocation or Galerkin, Finite
Elements or Spectral type, to Stochastic Collocation, Stochastic Finite Elements, and
Monte Carlo methods. We chose to present neural field equations in their simplest
form, with a single neuronal population, as we believe this is the case of interest
when one aims to study theoretical convergence properties of numerical schemes. Our
theory is essentially applicable without modification to neural mass models, which
correspond to a finite-dimensional version of neural fields, with $\XSet = \RSet^n$,
and in which integral operators are replaced by matrix-vector multiplications. The
result we presented in the projected equation covers this case.

Rather than deriving a theory for neural fields with $p$ populations, which
amounts to a switch in function spaces (for instance from $\XSet = L^2(D)$ to $\XSet =
\bigl(L^2(D)\bigr)^p$, see \cite{faugeras2008absolute} for an example) we chose to present
jointly linear and nonlinear neural fields with one population, because we envisage
the latter to be more relevant in the context of numerical analysis. An interesting
future \revision{extension} of the present theory concerns neuronal networks of second generation,
which are derivable as exact limits of networks of spiking, quadratic
integrate-and-fire neurons
\cite{montbrioMacroscopicDescriptionNetworks2015,
laingExactNeuralFields2015,coombesNextGenerationNeural2016}.
Spatially-extended, continuous models of this type have been proposed in literature
\cite{Esnaola-Acebes.2017,byrne2019next,laingInterpolatingBumpsChimeras2021}, and
they overcome some biological limitations of neural fields of first generation, as
firing rates are an emergent feature in these models. Even though these models are
nonlinear, nonlocal and similar in structure to the ones studied here, our theory can
not be immediately extended in that context, because their functional analytical
setup and well-posedness are still unavailable in literature. We expect, however,
that similar arguments to the ones used here should be valid for second generation
neural fields too. Further, the present work opens up the possibility of studying
problems in which the neural field equations involve or are coupled to diffusion and
reaction processes. This occurs, for instance, when metabolic or dendritic processes
are included in the model
\cite{avitabileNumericalInvestigationNeural2020,
idumah2023,avitabileWellPosednessRegularitySolutions2024}. It seems now possible to
study coupled problems by combining our results to the ones available for
elliptic and parabolic PDEs \cite{babuskaStochasticCollocationMethod2007,Zhang.2012}.

\section*{Acknowledgements} We are grateful to Jan van Neerven for helpful
discussions. 
\revision{
We are particularly grateful to an anonimous reviewer for pointing out a
  shortcut to prove \cref{thm:volterra} more elegantly and concisely 
(our first attempt is accessible at \url{https://arxiv.org/abs/2505.16343v2}).
} DA and FC acknowledge support from the National Science Foundation
under Grant No. DMS-1929284 while the authors were in residence at the Institute
for Computational and Experimental Research in Mathematics in Providence, RI,
during the “Math + Neuroscience: Strengthening the Interplay Between Theory and
Mathematics" programme.

\appendix
\label{sec:appendix}

\section{Additional proofs}
\begin{proof}[Proof of \cref{prop:HMapping}]\label{proof:HMapping}
  Fix $k \in \WSet$. The compactness of $H(k)$ is proved in 
  \cite[Sections 1.2.1 and 1.2.3]{atkinson1997}. The bound \cref{eq:TBound} is a
  rewriting of \cite[Lemma 2.6]{avitabileProjectionMethodsNeural2023}, see also \cite[Equation
  1.2.21 and Equation 1.2.34]{atkinson1997}. The linear mapping $H$ is therefore
  continuous on $\WSet$ to $K(\XSet)$. 

  Further, the continuous image of a separable
  topological space is separable \cite[Theorem 16.4.a]{willardGeneralTopology2004},
  hence the separability of $\WSet$ implies the separability of $H(\WSet)$. We
  provide below an argument that is specific to normed spaces, and to the $H$ defined
  above 
Since $\WSet$ is a separable metric space, it contains a countable dense subset $W_0
\subset \WSet$. Consider the subset $H_0 = \{ H(w) \colon w \in W_0\} \subset
H(\WSet)$. Since $W_0$ is countable, then $H_0$ is countable. We claim that $H_0$ is
also dense in $H$. Since $W_0$ is dense in $\WSet$, then for any $w \in \WSet$ and $\eps
>0 $ there exists $w_0 \in W_0$ such that $ \| w - w_0 \|_\WSet < \eps$. Fix $h \in
H(\WSet)$ and $\eps > 0$, then $h = H(w)$ for some $w \in \WSet$; let $h_0 = H(w_0) \in
H_0$ and estimate
\[
  \| h - h_0 \|_{BL(\XSet)} \leq \| w - w_0 \|_{\XSet} < \eps,
\]
hence for any $h \in H(\WSet)$ and $\epsi > 0$ there is an $h_0 \in H_0$ such that $
\| h-h_0 \|_{BL(\XSet)} < \epsi$, therefore $H_0$ is dense in $H(\WSet)$. Since $H_0$
is a countable dense subspace of the metric space $H(\WSet)$, then $H(\WSet)$ is
separable.
\end{proof}

\begin{proof}[Proof of \cref{prop:kappaEst}]\label{proof:kappaEst}
  \textit{Statement 1}. Since $\| \blank \|_{\WSet} \colon \WSet \to \RSet$ is a norm, 
  and $\omega \mapsto w(\blank,\blank,\omega)$ is strongly $\prob_w$-measurable on
  $\Omega$ to $\WSet$ by \cref{hyp:randomData}.\labelcref{hyp:kernel}, then
  \cite[Corollary 1.1.24]{hytonenAnalysisBanachSpaces2016a} guarantees that
  the composition mapping $\omega \mapsto \| w(\blank,\blank,\omega)\|_{\WSet}$ is a
  strongly $\prob_w$-measurable random variable (similar considerations apply for the
  random variables $\kappa_g$, $\kappa_f$, $\kappa_{f'}$, and $\kappa_{v}$ appearing
  in other statements of this proposition). 

  Further,  $H \colon \WSet \to K(\XSet) \subset BL(\XSet)$ is continuous by
  \cref{prop:HMapping}, and hence measurable, and a further application of \cite[Corollary
  1.1.24]{hytonenAnalysisBanachSpaces2016a} gives the strong measurability
  of the composition $W(\omega) = H(w(\blank,\blank,\omega))$, hence $W(\omega)$ is a
  strongly $\prob_w$-measurable $H(\WSet)$-valued random variable. The estimate 
  in statement 1 follows from the estimate in \cref{prop:HMapping}.

  \textit{Statements 2 to 4}. These statements follow directly from the definitions
  on norms and \cref{hyp:randomData} (see also proof of Statement 1 above).
  
  \revision{
  \textit{Statement 5}. We fix $u \in \YSet := C^0(J,\XSet)$ and prove
  the existence of a sequence of simple functions $\{ \lambda_k \}_k$
  converging $\prob_f$-almost surely to $\lambda$ in $\YSet$. By \cref{hyp:randomData} there
  exists a sequence of functions $\{ f_k(\blank,\omega) \}_k \subset BC^1(\RSet)$ of
  the form
  \[
    f_k (\blank, \omega) = \sum_{k'=1}^{n_k} 1_{A_{kk'}}(\omega) \xi_{kk'},
    \qquad A_{kk'} \in \calF_f, \qquad \xi_{kk'} \in BC^1(\RSet),
    \qquad k' \in \NSet_{n_k} \qquad k \in \NSet,
  \]
  and such that $\| f(\blank,\omega) - f_k(\blank,\omega)\|_{BC^1(\RSet)} \to 0$ 
  as $k \to \infty$ for almost all $\omega_f \in \Omega_f$. This implies that the
  operators $F_k \colon \XSet \times \Omega \to \XSet$ defined by $F_k(v,\omega)(x) =
  f_k(v(x),\omega)$ satisfy
  % , for any $v \in \XSet$
  \[
    \| F( v, \omega) - F_k(v,\omega)\|_{\XSet} \leq 
    \kappa_D \| f(\blank,\omega) - f_k(\blank,\omega)\|_{BC(\RSet)} \to 0 \qquad \text{$\prob_f$-almost surely for all
    $v \in \XSet$.}
  \]
  We let
  \[
    \lambda_k(\omega)(t)(x) := F_k(u(t),\omega)(x)
    % f_k\bigl(u(t)(x),\omega\bigr) = 
    = \sum_{k'=1}^{n_k} 1_{A_{kk'}}(\omega) \; \xi_{kk'}\bigl(u(t)(x)\bigr),
    % \qquad A_{kk'} \in \calF_f, \qquad \xi_{kk'},
    \qquad \lambda_k(\omega) \in \YSet,
    \qquad k \in \NSet,
  \]
  and derive, for almost all $\omega \in \Omega_f$
  \[
    \| \lambda(\omega) - \lambda_k(\omega)\|_{\YSet} = \max_{t \in J} \|
    F(u(t),\omega) - F_k(u(t),\omega) \|_{\XSet} 
    \leq \kappa_D \| f(\blank,\omega) - f_k(\blank,\omega)\|_{BC(\RSet)} \to 0.
  \]
  }
  \textit{Statement 6}. To prove estimate \cref{eq:LipN} in the linear
  case ($\LSet=1$) note that the hypotheses imply that for almost all $(\omega_w,
  \omega_g) \in \Omega_w \times
  \Omega_g$, the mapping $(t,u) \mapsto -u + W(\omega_w) u + g(t , \omega_g)$ is on
  $J \times \XSet$ to $\XSet$. To show continuity, consider a sequence $\{ (t_k,u_k) \}$
  converging to $(t,u) \in J \times \XSet$, and note that if $\LSet = 1$
  \begin{multline}
    \| N(t_k,u_k,\omega_w,\omega_g) - N(t,u,\omega_w,\omega_g) \|_\XSet 
    \leq (1+ \kappa_w(\omega_w) ) \| u_k - u \|_\XSet \\ + \|g(t_k, \omega_g) - g(t,
    \omega_g)\|_\XSet
  \end{multline}
  can be made arbitrarily small for almost all $(\omega_w,\omega_g) \in \Omega_v \times
  \Omega_g$, by taking $k$ sufficiently large, owing to the
  convergence of $u_k$ to $u$, and to the continuity of $t \mapsto g(t ,
  \omega_g)$. Setting $t_k = t$ in the previous bound, proves the Lipschitz
  condition for $\LSet = 1$. Estimate \cref{eq:LipN} for the nonlinear case $\LSet = 0$ holds
  because the hypotheses of \cite[Lemma 2.7]{avitabileProjectionMethodsNeural2023} 
  hold for almost all $(\omega_w,\omega_f,\omega_g)$. The estimate \cref{eq:BoundN}
  is a direct consequence of the triangle inequality and Statements 1--4.
\end{proof}

\begin{proof}[Proof to \cref{cor:subSigma}]\label{proof:subSigma}
  Let us consider first the linear case, $\LSet = 1$. \Cref{thm:LpRegularity}
  implicitly assumes that $w$, $g$, $v$ are $\calF_w$-, $\calF_g$-
  $\calF_v$-measurable functions, and this resulted in $u$ being $\calF_w \times
  \calF_g \times \calF_v$-measurable.
  In passing, we note that such $\sigma$-algebras have been omitted in the notation
  of the function spaces for simplicity, but will be reinstated below for the sub
  $\sigma$-algebras. If $w$ is $\calG_w$-measurable, $g$ $\calG_g$-measurable, and $v$
  $\calG_v$-measurable, then $w \in L^p(\Omega_w,\calG_w,\WSet)$, $g \in
  L^p(\Omega_g,\calG_g,C^0(J,\XSet))$, and $v \in L^p(\Omega_v,\calG_v,\XSet)$,
  respectively. \cref{thm:LpRegularity} can be applied with $\calF_\alpha$ replaced
  by $\calG_\alpha$, and we conclude that there exists a solution, say $z \in L^p(\Omega,\calG,C^1(J,\XSet))$ to
  \cref{eq:NRNFOp}. But the solution to \cref{eq:NRNFOp} is unique, hence $u = z$ is
  $\mathcal{G}$-measurable. The proof for the nonlinear case follows
  almost identical steps of the linear one, and we omit it.
\end{proof}

\begin{proof}[Proof of \cref{lem:finDimNoise}]\label{proof:finDimNoise}
  We present a proof for the nonlinear case $\LSet = 0$, and omit the
  one for $\LSet =1$, which is almost identical. For each $\alpha \in \USet = \{w,f,g,v\}$,
  consider the probability space $(\Omega_\alpha,\sigma_\alpha(Y_\alpha),
  \prob_\alpha)$ where $\sigma_\alpha$ is the $\sigma$-algebra generated by
  $Y_\alpha$, that is, 
  $
  \sigma_\alpha(Y_\alpha) = \{ Y^{-1}_\alpha(B) \colon B \in
  \mathcal{B}(\RSet^{m_\alpha})\}
  $. Since $w(x,x',\blank)$ depends on $Y_w$, it is
  $\sigma_w(y_w)$-measurable, and similar considerations hold for $f$, $g$, and
  $v$. By \cref{def:finDimNoise}, $\sigma_\alpha(y_\alpha) \subset \calF_\alpha$ is a
  $\sigma$-algebra. We can therefore apply \cref{thm:LpRegularity} for $\calG_\alpha =
  \sigma_\alpha(Y_\alpha)$, $\alpha \in \USet$, and we conclude that $u(x,t,\blank)$ is
  $\sigma(Y)$-measurable, for all $(x,t) \in D \times J$, where $\sigma(Y) =
  \times_\alpha \sigma_\alpha(Y_\alpha)$. By the Doob--Dynkin Lemma (see
  \cite[Lemma 4.1]{martinezfrutosOptimalControlPDEs2018}) for any $(x,t) \in D \times
  J$ there exists a measurable function $h_{x,t} \colon \RSet^m \to \RSet$ such that
$u(x,t,y) = h_{x,t}(y)$, that is, $u(x,t,\omega) = h_{x,t}(Y(\omega)) =: \tilde
u(x,t,Y(\omega))$, which completes the proof.
\end{proof}

% \printProofs

\bibliography{references}
\bibliographystyle{siamplain}
 
\end{document}